\documentclass{amsart}

\usepackage{latexsym,amsfonts,amsmath,epsfig,tabularx,amsthm,dsfont,mathrsfs} 
\usepackage{enumerate} 
\usepackage{enumitem}
\usepackage{upref}

\newtheorem{Satz}{Satz}[section]
\newtheorem{Thm}[Satz]{Theorem}
\newtheorem{Lem}[Satz]{Lemma}
\newtheorem{Corol}[Satz]{Corollary}
\theoremstyle{remark}

\theoremstyle{definition} 
\newtheorem{Def}[Satz]{Definition}
\newtheorem{Rem}[Satz]{Remark}
\newtheorem{Prop}[Satz]{Proposition}
\newcommand{\wt}{\ensuremath{\widetilde}}
\newcommand{\hra}{\ensuremath{\hookrightarrow}}
\newcommand{\vr}{\ensuremath{\varrho}}
\newcommand{\ol}{\ensuremath{\overline}}

\newcommand{\ve}{\ensuremath{\varepsilon}}

\newcommand{\Lci}{\ensuremath{\overset{\circ}{L}}{}^r_{p}(\real^n)}

\newcommand{\Lr}{\ensuremath{L^r_{p}(\real^n)}}

\newcommand{\Hrow}{\ensuremath{H^{\vr} L_{p} (w,\rn)}}
\newcommand{\Lrp}{\ensuremath{{ L}^{r}_{p}(\rn)}}
\newcommand{\Lmwe}{\ensuremath{{ L}^{r}_{p}(w,\rn)}}
\newcommand{\eLmwe}{\ensuremath{{ L}^{r}_{p}(w,\real)}}

\newcommand{\Hrdw}{\ensuremath{H^{\vr} L_{p'} (w,\rn)}}
\newcommand{\eHrdw}{\ensuremath{H^{\vr} L_{p'} (w,\real)}}

\newcommand{\Lciw}{\ensuremath{\overset{\circ}{L}}{}^r_{p}(w,\real^n)}
\newcommand{\eLciw}{\ensuremath{\overset{\circ}{L}}{}^r_{p}(w,\real)}

\newcommand{\LmwztN}{\ensuremath{{ L}^{{r p_0}}_{\frac{p}{p_0}}(w,\rn)}}
\newcommand{\LmwztNP}{\ensuremath{H^{\vr} L_{\left(\frac{p}{p_0}\right)'} (w,\rn)}}
\newcommand{\Lpw}{\ensuremath{L_{p,w} (\rn)}}
\newcommand{\Lpzw}{\ensuremath{L_{p_0,w} (\rn)}}
\newcommand{\Lpzwt}{\ensuremath{L_{p_1,w} (\rn)}}

\newcommand{\di}{\ensuremath{{\mathrm d}}}
\newcommand{\zn}{\ensuremath{{\mathbb Z}^n}}
\newcommand{\rn}{\ensuremath{{\mathbb R}^n}}
\newcommand{\real}{\mathbb{R}}
\newcommand{\rat}{\mathbb{Q}}
\newcommand{\nat}{\mathbb{N}}
\newcommand{\ganz}{\mathbb{Z}}

\newcommand{\eq}{equation}
\newcommand{\supp}{\ensuremath{\mathrm{supp \,}}}

\newcommand{\norm}[2]{\left\Vert\left. #1 \right| #2 \right\Vert}

\begin{document}
\title[Boundedness results in Morrey spaces via extrapolation and duality]{\textbf{The boundedness of operators in\\Muckenhoupt weighted Morrey spaces\\via extrapolation techniques and duality}} 


\author{
Marcel Rosenthal \& Hans-J{\"u}rgen Schmeisser} 
\address{Universidad del Pa\'{i}s Vasco/Euskal Herriko Unibertsitatea, Departamento de Matemáticas/Matematika saila\\
Apdo. 644, 48080 BILBAO, Spain}
\email{marcel.rosenthal@uni-jena.de}
\address{Universit\"at Jena, Mathematisches Institut\\
Ernst-Abbe-Platz 2, 07743 Jena, Germany}
\email{hans-juergen.schmeisser@uni-jena.de} 

\thanks{The first author is supported by the German Academic Exchange Service (DAAD)}

\subjclass[2010]{Primary 42B35, 46E30, 42B15, 42B20; Secondary 42B25}

\date{July 15, 2015.} 

\keywords{Morrey spaces, predual Morrey spaces, extrapolation, Muckenhoupt weights, singular integral operators, Calder\'{o}n-Zygmund operators, H\"ormander-Mikhlin type multipliers, Marcinkiewicz multipliers, maximal Carleson operator, commutators, associated Morrey spaces} 

\begin{abstract}
The bidual of the closure of smooth functions with respect to the Morrey norm coincides with the Morrey space. This assertion is generalized to some Muckenhoupt weighted Morrey spaces. 
We combine this fact with basic extrapolation techniques due to Rubio de Francia adapted to weighted Morrey spaces. This leads to new results on the boundedness of operators even for the unweighted case.
\end{abstract}

\maketitle

\section{Introduction}
An important tool of modern harmonic analysis is the extrapolation theorem due to Rubio de Francia. 
For a given operator $T$ we suppose that for some $p_0$, $1 \le p_0 < \infty$, and
for every weight belonging to the Muckenhoupt class $A_{p_0}$ the inequality
\begin{\eq} \label{i1}
  \norm{Tf}{\Lpzw}= \left(\int_{\rn} |Tf(x)|^{p_0} w(x) \di x \right)^\frac{1}{p_0} \le c \norm{f}{\Lpzw}
\end{\eq}
holds, where the constant $c$ is independent of $f$ but can depend on $[w]_{A_{p_0}}$.
Then for every $p$, $1 < p < \infty$, and every $w \in A_{p}$ there exists a constant depending
on $[w]_{A_{p}}$ such that
\begin{\eq} \label{i2}
  \norm{Tf}{\Lpw}\le c \norm{f}{\Lpw}
\end{\eq}
(cf. \cite{Ru1, Ru2, CMP11, D13}).
In this paper we consider Muckenhoupt weighted Morrey spaces $\Lmwe$ collecting all locally integrable complex-valued functions given on $\rn$ with 
\begin{align*}
  		\left\|f|\Lmwe\right\|&= 
\sup_{J\in\ganz, M\in\ganz^n} w(Q_{J,M})^{-\left(\frac{1}{p}+\frac{r}{n} \right)} \left(\int_{Q_{J,M}} |f(x)|^p w(x) \di x \right)^\frac{1}{p}
<\infty
\end{align*}
where $Q_{J,M}=2^{-J}\left(M+\left[-1,1\right]^n\right)$,  $1 < p < \infty$ and $-\frac{n}{p} \le r < 0$. The spaces $\Lmwe$ coincide with the unweighted Morrey spaces $\Lr$ if $w(x)= 1$ 
and furthermore with $\Lpw$ if $r=-\frac{n}{p}$.
It is our aim to show that \eqref{i1} also implies
that for every $p$, $1 < p < \infty$, every $r$, $-\frac{n}{p} \le r < 0$, and every $w \in A_{p}$ there exists a constant depending
on $[w]_{A_{p}}$ such that
\begin{\eq} \label{i3}
  \norm{Tf}{\Lmwe}\le c \norm{f}{\Lmwe} 
\end{\eq}
for $D(\rn)$. 
If we assume in addition that $T$ admits an unique and continuous extension from $D(\rn)$ on $\Lciw$ where $\Lciw$ stands for the completion of $D(\rn)$ with respect to $\left\|\cdot|\Lmwe\right\|$, then we achieve 
\begin{\eq} \label{i4}
  T: \Lciw \hra \Lmwe.
\end{\eq}
Note that linearity of $T$ is sufficient for the existence of this unique and continuous extension but also other operators as the maximal operator or maximally truncated singular integral operators are admissible (cf. \eqref{ToY} for another sufficient condition on $T$).
This is a new method to prove \eqref{i4} even in the unweighted Morrey spaces $\Lr$. 
In many previous papers inequalities of type \eqref{i4} are deduced assuming the estimate
\begin{\eq} \label{i5}
  |(Tf)(y)|\leq c \int_{\real^n} \frac{|f(x)|}{|y-x|^n} \di x \quad \text{ for all } f\in D(\rn) \text{ and } y\notin \supp(f)
\end{\eq}
and using the (weighted) $L_{p_0}(\rn)$ boundedness of $T$ 
(cf. \cite{N94,Y98, G11, Gul12, Mu12, RS14, PT15, Wan16} and the references given there or similar size estimates of the operator cf. \cite{Alv96, SFZ13}).
Sharpening \eqref{i4} to 
\begin{\eq} \label{i6}
  T: \Lciw \hra \Lciw
\end{\eq}
we find extensions of $T$ via (bi-)duality on $\Lmwe$ such that
\begin{\eq} \label{i7}
  T: \Lmwe \hra \Lmwe.
\end{\eq}
This is based on the duality result
\begin{\eq} \label{i8}
  (\Lciw')' \cong\Lmwe
\end{\eq}
which extends the corresponding unweighted duality assertion observed in \cite{AX12} and proved completely in \cite{RoT14_2}.
Our method proving \eqref{i4} and \eqref{i7} using \eqref{i1} instead of \eqref{i5} leads to new results on the boundednesses of operators even in the unweighted Morrey spaces $\Lr$. 
As examples H\"ormander-Mikhlin type multipliers, Marcinkiewicz multipliers and commutators will be considered.
Let us mention that in many related papers (starting from Peetre \cite{Pee66} and many following scholars) dealing with various generalizations of $\Lr$ 
the non-density of $D(\rn)$ in $\Lr$ is not taken into account. Then the use of \eqref{i5} for all $f\in \Lr$ (instead of $f\in D(\rn)$) has to be justified, in particular, for singular integrals and multipliers. Moreover, one has also to clarify in which way one extends the operator $T$ given on some $L_{p_0}(\rn)$ space (or given on $D(\rn)$) to $\Lr$ (whenever necessary cf. singular integrals and multipliers). 
For a detailed discussion we refer to the forerunner results \cite{Alv96, AX12, RT13, Ros13, RoT14_2, T-HS, RS14, Ad15}.
Based on \eqref{i6} and \eqref{i7} we are able to complete previous results for singular integrals (Calder\'{o}n-Zygmund operators), multipliers, commutators, \ldots with respect to \eqref{i4}. 
In contrast to \cite{Alv96, AX12, Ad15} we are working on $\Lciw$ (instead on $\Lciw'$). \cite{AX12, Ad15} use also \eqref{i1} to obtain partial results in the unweighted situation $\Lrp$.
By the fact that there are also negative results with respect to interpolation of Morrey spaces (cf. \cite{BRV99}) it was not clear that extrapolation will work.
\par
After introducing the notation and some preliminaries in Section 2, 
Section 3 is concerned with preparations which are needed to prove the duality result \eqref{i8} in Section 4.
Here we investigate basic embeddings, density and separability of weighted (pre-)dual Morrey spaces. 
Using duality in Morrey spaces, \eqref{i8}, we prove \eqref{i3} in Section 5 generalizing ideas of \cite[Theorem 4.6]{CMP11} and \cite{CGCMP06}. 
In Section 6 we present the main result of the paper 
which states that \eqref{i1} and the existence of an unique and continuous extension of $T$ from $D(\rn)$ to $\Lciw$ imply \eqref{i6} and \eqref{i7} (cf. Theorem \ref{vB:thm}). Moreover, if $T$ is formally self-adjoint in $\Lpzw$, then we have also
\[
  T: \Lciw' \hra \Lciw'
\]
where $\Lciw'$ admits an atomic characterization (cf. Definition \ref{D2.3} and Theorem \ref{TDp2:GG}). We shall apply our general results to distinguished operators (Calder\'{o}n-Zygmund operators, H\"ormander-Mikhlin type multipliers, Marcinkiewicz multipliers, maximal Carleson operator, commutators). Finally, we explain how one can lift these results to the vector-valued case. In the last section we characterize the associated spaces of the Morrey spaces.
\section{Notation, Morrey spaces and preliminaries}
We use standard notation. Let $\nat$ be the collection of all natural numbers and 
$\rn$ be the Euclidean $n$-space, where
$n\in \nat$. Put $\real = \real^1$. 
Let $S(\rn)$ be the Schwartz space of all complex-valued rapidly decreasing infinitely differentiable functions on $\rn$. 
Let $D(\rn)$ be the collection of all infinitely differentiable functions 
with compact support in $\rn$, where the support of a function $f$ is abbreviated by $\supp(f)$.
Furthermore, $L_{p,w} (\rn)$ with $1 \le p <\infty$ is the complex Banach space of functions whose $p$-th power is integrable with respect to the weight $w:\rn\rightarrow [0,\infty]$ and which is 
normed by
\[
\| f \, | L_{p,w} (\rn) \| = \Big( \int_{\rn} |f(x)|^p \,  w(x) \di x \Big)^{1/p}.
\]
Moreover, we write $w(M)=\int_{M} w(x) \di x$ for the measurable subset $M$ of $\rn$. We similarly define $L_{p,w} (M)$.
If $w(x)=1$, we simply write $L_{p}(M)$, $\| \cdot \, | L_{p} (M) \|$ and $|M|$. 
Furthermore, 
$\chi_M$ denotes the characteristic function of $M$. As usual, $\ganz$ is the collection of all integers; and $\zn$ where $n\in \nat$ denotes the lattice of all points $m= (m_1, \ldots, m_n) \in \rn$ with $m_j \in \ganz$. 
Moreover, $L_1^{\mathrm{loc}}(\rn)$ collects all equivalence classes of almost everywhere coinciding measurable complex locally integrable functions, hence $f\in L_1 (M)$ for any bounded measurable set $M$ of $\rn$. 
If $Q$ denotes a cube in $\rn$ (whose sides are parallel to the coordinate axes), then $dQ$ stands for the concentric cube with side-length $d>1$ times of the side-length of $Q$.
For any $p\in (1,\infty)$ we denote by $p'$ the conjugate index, namely, $1/p+1/{p'}= 1$.
For Banach spaces $X$ and $Y$ we denote by 
\[T:X\hra Y\]
a bounded operator mapping $X$ into $Y$.
That is, we have
\[
  \left\|Tx|Y\right\|\le c \left\|x|X\right\| 
\]
for all $x\in X$
where the constant $c$ is independent of $x$.
The concrete value of constants may vary from one
formula to the next, but remains the same within one chain of (in)equalities. 
Finally, $A \cong B$ means that there are two constants $c$, $C > 0$ such that $c A \le B \le C A$.
\begin{Def} 
We say that 
$w\in L_1^\text{loc}(\rn)$ with $w>0$ almost everywhere belongs to $A_p$ for $1<p<\infty$ if
\begin{\eq} \label{Ap}
  [w]_{A_p}\equiv \sup_Q \frac{w(Q)}{|Q|} \left( \frac{w^{1-p'}(Q)}{|Q|} \right)^{p-1}<\infty,
\end{\eq}
where the supremum is taken over all cubes $Q$ in $\real^n$ (whose sides are parallel to the coordinate axes). 
The value $[w]_{A_p}$ is called the $A_p$ constant of the \textit{Muckenhoupt weight} $w$.
\end{Def}
We define some Muckenhoupt weighted Morrey spaces.
\begin{Def}{} \label{d1:def}
\upshape 
	For $1< p< \infty$, $-\frac{n}{p}\leq r<0$ and $w\in A_p$ we define 
	\begin{equation*}
				\Lmwe\equiv\{f \in L_1^{\text{loc}}(\real^n) \mbox{ : } \left\|f|\Lmwe\right\|<\infty\}
	\end{equation*}
	with the norm  
	\begin{align*} 
		\left\|f|\Lmwe\right\|\equiv & 
\sup_{J\in\ganz, M\in\ganz^n} w(Q_{J,M})^{-\left(\frac{1}{p}+\frac{r}{n} \right)} \left\|f|L_{p, w}(Q_{J,M})\right\|.
	\end{align*}	
Recall that $Q_{J,M}=2^{-J}\left(M+\left[-1,1\right]^n\right)$, $J\in\ganz$, $M\in\ganz^n$ is a dyadic cube with side length $2^{-J+1}$ centered at $2^{-J}M$.
\end{Def}
\begin{Rem}
We observe that ${L}^{-{n}/{p}}_p(w,\rn)=L_{p, w}(\rn)$. If $w(x) \di x = \di x$, $\Lmwe$ coincides with the unweighted Morrey spaces $\Lrp$.
Furthermore, we mention that the a priori assumption $f \in L_1^{\text{loc}}(\real^n)$ can be omitted.  
Indeed, if $f$ is a measurable complex function defined on $\rn$, then $\left\|f|\Lmwe\right\|<\infty$, H\"older's inequality and \eqref{Ap} yield
\begin{align*}
  \int_{Q_{J,M}} |f(x)|\di x 
	&\le \left\|f|L_{p, w}(Q_{J,M})\right\|  {w^{1-p'}(Q_{J,M})}^{\frac{1}{p'}}
	\\&\le [w]_{A_p}^{\frac{1}{p}} \left\|f|L_{p, w}(Q_{J,M})\right\| \frac{|Q_{J,M}|}{w(Q_{J,M})^\frac{1}{p}} 
	\\&\le [w]_{A_p}^{\frac{1}{p}} \left\|f|\Lmwe\right\| w(Q_{J,M})^\frac{r}{n} <\infty
\end{align*}
for all $J\in\ganz$, $M\in\ganz^n$. Hence, $f \in L_1^{\text{loc}}(\real^n)$.
\end{Rem}
\begin{Rem}
Let $1<p<\infty$ and $w\in A_p$. 
Then for a cube $Q$ in $\rn$ (whose sides are parallel to the coordinate axes) and a measurable subset $S\subset Q$ we have
\begin{equation} \label{DC} 
  \frac{w(Q)}{w(S)} \le c \left(\frac{|Q|}{|S|}\right)^p
\end{equation} 
where the constant $c$ does not depend on $Q$ and $S$ cf. \cite[(7.3)]{D01}.
Moreover, $w$ satisfies the reverse doubling condition, i.e. there exists a constant $0<c<1$ such that
\begin{equation} \label{RDC} 
  w(Q)\le c\, w(2 Q)
\end{equation} 
holds for arbitrary cubes $Q$ in $\rn$ (whose sides are parallel to the coordinate axes) where the constant $c$ does not depend on the center and side-length of the cubes $Q$ cf. \cite[Lemma 7.5]{D01}. 
\end{Rem}
\begin{Lem}  \label{p4.2}
Let $1< p< \infty$, $-\frac{n}{p}\leq r<0$ and $w\in A_{p}$. 
Then 
\begin{align}
  \norm{f}{\Lmwe} 
	\cong & \sup_{J\in\ganz, M\in\ganz^n} w(\wt{Q_{J,M}})^{-\left(\frac{1}{p}+\frac{r}{n} \right)} \left\|f|L_{p, w}(\wt{Q_{J,M}})\right\| \label{Meq1}
	\\ \cong &  \sup_{x\in\rat^n, J\in\ganz} w(Q(x,J))^{-\left(\frac{1}{p}+\frac{r}{n} \right)} \left\|f|L_{p, w}(Q(x,J))\right\| \label{Meq2}
\end{align}
for all $f\in L_1^\text{loc}(\rn)$ where $\wt{Q_{J,M}}=2^{-J}\left(M+\left[0,1\right]^n\right)$ and $Q(x,J)\equiv x+ 2^{-J}\left[-1,1\right]^n$.
\end{Lem}
\begin{proof}
	We observe
	\begin{align*}
	  & w({Q_{J,M}})^{-\left(\frac{1}{p}+\frac{r}{n} \right) p} \left\|f|L_{p, w}({Q_{J,M}})\right\|^p 
		\\ \le & 2^{n} \sup_{\wt{M}\in\ganz^n:\ \wt{Q_{J,\wt{M}}} \subset Q_{J,M}}
		w(\wt{Q_{J,\wt{M}}})^{-\left(\frac{1}{p}+\frac{r}{n} \right)p} \left\|f|L_{p, w}(\wt{Q_{J,\wt{M}}})\right\|^p
	\end{align*}
		and hence
		\[
		  w({Q_{J,M}})^{-\left(\frac{1}{p}+\frac{r}{n} \right)} \left\|f|L_{p, w}({Q_{J,M}})\right\| \le c \sup_{J\in\ganz, M\in\ganz^n} w(\wt{Q_{J,M}})^{-\left(\frac{1}{p}+\frac{r}{n} \right)} \left\|f|L_{p, w}(\wt{Q_{J,M}})\right\|
		\]
		where the constant $c$ does not depend on $J\in\ganz$ and $M\in\ganz^n$.
	Furthermore, $\wt{Q_{J,M}}\subset Q_{J,M}$ yields
\begin{align*}
	  & w(\wt{Q_{J,M}})^{-\left(\frac{1}{p}+\frac{r}{n} \right)} \left\|f|L_{p, w}(\wt{Q_{J,M}})\right\|
		\\ \le &
		\left(\frac{w({Q_{J,M}})}{w(\wt{Q_{J,M}})}\right)^{\frac{1}{p}+\frac{r}{n} }
		w({Q_{J,M}})^{-\left(\frac{1}{p}+\frac{r}{n} \right)} \left\|f|L_{p, w}({Q_{J,M}})\right\|.
	\end{align*}	
Since 
\[
  \frac{w({Q_{J,M}})}{w(\wt{Q_{J,M}})}
	\le c 
\]
where the constant $c$ does not depend on $J\in\ganz$ and $M\in\ganz^n$ (cf. \eqref{DC}), we obtain the equivalence \eqref{Meq1}.
Moreover, 
\[
  \norm{f}{\Lmwe} 
	\le  \sup_{x\in\rat^n, J\in\ganz} w(Q(x,J))^{-\left(\frac{1}{p}+\frac{r}{n} \right)} \left\|f|L_{p, w}(Q(x,J))\right\| 
\]
is obvious.
Now, we choose a cube $Q(x,J)$. We fix then $\wt{M}\in \ganz$ such that $\wt{M} 2^{-J+1} \le x_j < (\wt{M}+1) 2^{-J+1}$ where $x=(x_1, \ldots, x_n)$ and define
	\begin{\eq} \label{WFM1}
	  M_j \equiv \begin{cases} \wt{M}, & x_j-\wt{M} 2^{-J+1}\le 2^{-J} 
		\\ \wt{M}+1, & (\wt{M}+1) 2^{-J+1} -x_j < 2^{-J}  
		\end{cases}
	\end{\eq}
	for $j=1,\ldots,n$. Therefore, $Q(x,J)\subset Q_{J-1,M}$ where $M = (M_1,\ldots, M_n)$ and
	\begin{align*} 
  & w(Q(x,J))^{-\left(\frac{1}{p}+\frac{r}{n} \right)} \left\|f|L_{p, w}(Q(x,J))\right\|
	\\ \le & \left(\frac{w({Q_{J-1,M}})}{w(Q(x,J))}\right)^{\frac{1}{p}+\frac{r}{n} }
	w(Q_{J-1,M})^{-\left(\frac{1}{p}+\frac{r}{n} \right)} \left\|f|L_{p, w}(Q_{J-1,M})\right\| .
	\end{align*}
	Using \eqref{DC} the equivalence \eqref{Meq2} follows.
\end{proof}
\begin{Rem}
By means of \eqref{DC} one can show analogously 
\begin{align}
  \norm{f}{\Lmwe} \cong &  \sup_{x\in\rn, R>0} w(B_R(x))^{-\left(\frac{1}{p}+\frac{r}{n} \right)} \left\|f|L_{p, w}(B_R(x))\right\| \label{Meq3}
	\end{align}
where $B_R(x)$ stands for the ball with radius $R$ centered at $x\in\rn$.	
\end{Rem}
%
%
%
%
\section{Non-separability, density and embeddings of weighted Morrey spaces}
\begin{Prop} \label{p2.5}
Let $1< p\le \wt{p}< \infty$, $-\frac{n}{p}\leq -\frac{n}{\wt{p}}\leq r<0$ and let $w\in A_{p}$. 
Then \[D(\rn) \hra S(\rn)\hra L_{u,w}(\real^n)\hra {{ L}^{r}_{\wt{p}}(w,\rn)} \hra \Lmwe\] where $u\equiv -\frac{n}{r}$. 
\end{Prop}
\begin{proof}
Let $f\in { L}^{r}_{\wt{p}}(w,\rn)$ and $Q\equiv Q_{J,M}$. 
H\"older's inequality 
yields
\begin{align*}
  w(Q)^{-\left(\frac{1}{p}+\frac{r}{n} \right)} \left\|f|L_{p, w}(Q)\right\|
	&\le w(Q)^{-\frac{1}{p}+\frac{r}{n} } \left\|f|L_{\wt{p}, w}(Q)\right\| w(Q)^{\left(1-\frac{p}{\wt{p}} \right)\frac{1}{p}} 
	\\&\le w(Q)^{-\frac{1}{\wt{p}}+\frac{r}{n} } \left\|f|L_{\wt{p}, w}(Q)\right\|
	\le \norm{f}{\Lmwe}
\end{align*}
and hence ${{ L}^{r}_{\wt{p}}(w,\rn)} \hra \Lmwe$. Thus, observing $L_{u,w}(\real^n)=L^r_{-n/r,w}(\real^n)$ and $1<\wt{p} \le -n/r$ we obtain the assertion.
\end{proof}
The following proposition is a generalization of \cite[Prop. 3.7]{RoT14_2} to weighted Morrey spaces. 
\begin{Prop}   \label{P3.7}
Let $1< p< \infty$, $-\frac{n}{p}< r<0$, $w\in A_{p}$ and let $u\equiv -\frac{n}{r}$.
Then the spaces $\Lmwe$ are non-separable. Furthermore, neither $D(\rn)$ nor $S(\rn)$ 
nor $L_{u,w} (\rn)$ nor $L_\infty^\text{comp}(\rn)$ is dense in $\Lmwe$ where $L_\infty^\text{comp}(\rn)$ denotes the collection of all compactly supported almost everywhere bounded functions. 
\end{Prop}
\begin{proof}
{\em Step 1.} First we prove the non-separability. Let
\begin{\eq}   \label{3.24}
Q_{l} \equiv 
2^{-l} \left( (2,\ldots,2) + [0,1]^n\right), \qquad l \in \nat,\ l\ge 2,
\end{\eq}
be disjoint cubes with $Q_{l} \subset Q = [-1,1]^n$ and let $2_n\equiv(2,\ldots,2)$. 
Let
\begin{\eq}    \label{3.25}
f^\lambda \equiv \sum^\infty_{l=2} \lambda_{l} \, w(Q_{l})^{\frac{r}{n}} \, \chi_{l}
\end{\eq}
where $\chi_{l}$ is the characteristic function of $Q_{l}$ and
\begin{\eq}  \label{3.26}
\lambda = \{ \lambda_{l} \}^\infty_{l=2} \qquad \text{with either $\lambda_{l} =1$ or $\lambda_{l} = -1$}.
\end{\eq}
Let $J \in \ganz$ and $M\in \zn$. 
Taking into account the construction of $\{Q_l\}$ and $\wt{Q_{J,M}}=2^{-J}\left(M+\left[0,1\right]^n\right)$ we observe that there are just three cases: 
\begin{itemize}
\item $\wt{Q_{J,M}}$ does not intersect another cube of $\{Q_l\}$, 
\item one of the cubes of $\{Q_l\}$ has nonempty intersection with $\wt{Q_{J,M}}$ (either $Q_{l_0}\subset \wt{Q_{J,M}}$ or $\wt{Q_{J,M}}\subset Q_{l_0} $ for some ${l_0}\in\nat$, $l_0\ge 2$ and $\wt{Q_{J,M}} \cap Q_l = \emptyset$ for all $l\in\nat$ with $l\neq l_0$) 
\item or $\wt{Q_{J,M}}$ does intersect infinitely many cubes of the family $\{Q_l\}_l$, in that case we have $\left\{\bigcup_{l:l\ge l_0} Q_l\right\} \subset \wt{Q_{J,M}}$ and $\wt{Q_{J,M}}\cap \left\{\bigcup_{l:l< l_0} Q_l\right\} = \emptyset $ for an appropriate $l_0\in\nat_0$ with $l_0\ge 2$. 
\end{itemize}
At first we treat the case $\wt{Q_{J,M}}\subset Q_{l_0} $ and $\wt{Q_{J,M}} \cap Q_l = \emptyset$ for $l\neq l_0$.
For $\wt{Q_{J,M}}\subset Q_{l_0} $ we observe then
\begin{\eq}   \label{3.27a}
  w({\wt{Q_{J,M}}})^{-\left(\frac{1}{p}+\frac{r}{n} \right)p} \int_{\wt{Q_{J,M}}} |f^\lambda (x)|^p w(x) \di x 
	= \left(\frac{w({\wt{Q_{J,M}}})}{w({Q_{l_0}})}\right)^{-\frac{r}{n}p} \le 1
\end{\eq} 
by $r<0$.
In the case $Q_{l_0}\subset \wt{Q_{J,M}}$ and $\wt{Q_{J,M}} \cap Q_l = \emptyset$ for $l\neq l_0$ we obtain
\begin{\eq}   \label{3.27b}
  w({\wt{Q_{J,M}}})^{-\left(\frac{1}{p}+\frac{r}{n} \right)p} \int_{\wt{Q_{J,M}}} |f^\lambda (x)|^p w(x) \di x 
	= \left(\frac{w({Q_{l_0}})}{w({\wt{Q_{J,M}}})}\right)^{\left(\frac{1}{p}+\frac{r}{n} \right)p} \le 1
\end{\eq}
by $\frac{1}{p}+\frac{r}{n}>0$. It remains the case that $\left\{\bigcup_{l:l\ge l_0} Q_l\right\} \subset \wt{Q_{J,M}}$ and $\wt{Q_{J,M}}\cap \left\{\bigcup_{l:l< l_0} Q_l\right\} = \emptyset $ where $l_0\in\nat_0$ with $l_0\ge 2$. 
Then $\wt{Q_{J,M}}= \wt{Q_{l_0-2,0_n}}=2^{2-l_0}[0,1]^n$ and
\begin{\eq}   \label{3.27}
w({\wt{Q_{J,M}}})^{-\left(\frac{1}{p}+\frac{r}{n} \right)p} \int_{\wt{Q_{J,M}}} |f^\lambda (x)|^p w(x) \di x 
= \sum_{l:l\ge l_0} \left(\frac{w({Q_{l}})}{w(\wt{Q_{l_0-2,0_n}})}\right)^{\left(\frac{1}{p}+\frac{r}{n} \right)p}
<\infty,
\end{\eq}
where the last inequality holds by the fact that $w$ satisfies the reverse doubling condition.
Indeed, the reverse doubling condition \eqref{RDC} yields 
\begin{\eq}   \label{Mine}
  w({Q_{l,2_n}})< c w({2Q_{l,2_n}})
\end{\eq} 
for a constant $c<1$ which does not depend on $l\in\nat$, where we recall $Q_{l,2_n}=2^{-l} \left( (2,\ldots,2) + [-1,1]^n\right)$.
Thus,
\begin{align}
\begin{split} \label{TFCH}
  & \sum_{l:l\ge l_0} \left(\frac{w({Q_{l}})}{w(\wt{Q_{l_0-2,0_n}})}\right)^{\left(\frac{1}{p}+\frac{r}{n} \right)p}
	\le \sum_{l:l\ge l_0} \left(\frac{w({Q_{l,2_n}})}{w(\wt{Q_{l_0-2,0_n}})}\right)^{\left(\frac{1}{p}+\frac{r}{n} \right)p}
	\\ \le & \sum_{l:l\ge l_0} \left[ c^{l- l_0} \left(\frac{w(2^{l- l_0}{Q_{l,2_n}})}{w(\wt{Q_{l_0-2,0_n}})}\right)\right]^{\left(\frac{1}{p}+\frac{r}{n} \right)p}
	\\ \le & \left(\frac{{w({Q_{l_0-2,0_n}})})}{w(\wt{Q_{l_0-2,0_n}})}\right)^{\left(\frac{1}{p}+\frac{r}{n} \right)p} \sum_{l:l\ge l_0} c^{(l- l_0)\left(\frac{1}{p}+\frac{r}{n} \right)p}, 
\end{split}
\end{align}
where the last inequality holds by  
\[
2^{l- l_0}{Q_{l,2_n}} = 2^{-l} 2_n + 2^{-l_0} [-1,1]^n \subset Q_{l_0-2,0_n}=2^{2-l_0}[-1,1]^n 
\]
and the right-hand side of \eqref{TFCH} is finite by \eqref{DC} and $c<1$.
Hence $f^\lambda \in \Lmwe$ by \eqref{3.27a}-\eqref{3.27} and Lemma \ref{p4.2}. If $\lambda^1\equiv\{\lambda^1_l\}_l$ and $\lambda^2\equiv\{\lambda^2_l\}_l$ are two different
admitted sequences, then one has $\lambda^1_{l_0} =1$ and $\lambda^2_{l_0} = -1$ for some $l_0 \in \nat$ with $l_0\ge 2$ and moreover by Lemma \ref{p4.2}
\begin{align} \begin{split} \label{3.28}
&\| f^{\lambda^1} - f^{\lambda^2} \, | \Lmwe \| 
\\ \ge & c\, w({Q_{l_0}})^{-\left(\frac{1}{p}+\frac{r}{n} \right)}
\Big( \int_{Q_{l_0}} \big|(f^{\lambda^1} - f^{\lambda^2})(x)\big|^p \, w(x) \di x \Big)^{1/p} =2c.
\end{split}\end{align}
But the set of all these admitted functions $f^\lambda$ is non-countable, having the cardinality of $\real$. Then it follows from \eqref{3.28} that $\Lmwe$ is not separable.
\\[0.1cm]
{\em Step 2.}
The separable Lebesgue space $L_{u,w} (\rn)$ is continuously embedded into the non-separable space $\Lmwe$. This shows that this
embedding cannot be dense. Moreover, $D(\rn)$, $S(\rn)$ and $L_\infty^\text{comp}(\rn)$ are subsets of $L_{u,w} (\rn)$ and hence not dense in $\Lmwe$.
\end{proof}
\begin{Def} \label{24:GG}
Let $1< p< \infty$, $-\frac{n}{p}\leq r<0$ and let $w\in A_p$. Then $\Lciw$ is the completion of $D(\rn)$ in $\Lmwe$.
\end{Def}
\begin{Def}  \label{D2.3}
Let $1<p<\infty$, $-n <\vr <-n/p$ and let $w\in A_{p'}$ with $p'=\frac{p}{p-1}$. 
Then the spaces 
$\Hrow$ collect all $h\in S'(\rn)$ 
which can be represented as
\begin{align}   \label{2.9}
\begin{split}
&h= \sum_{J\in \ganz, M\in \zn} h_{J,M} \quad \text{in} \quad S'(\rn) \quad \text{with} \quad \supp h_{J,M} \subset {Q_{J,M}},
\end{split}
\end{align}
such that
\begin{\eq}    \label{2.10a}
\sum_{J\in \ganz, M \in \zn} w(Q_{J,M})^{-\left(\frac{1}{p} + \frac{\vr}{n}\right)} \| h_{J,M} w^{-\frac{1}{p'}} \, | L_{p} (\rn) \| <\infty.
\end{\eq}
Furthermore,
\begin{\eq}    \label{2.10b}
\| h \, | \Hrow \| \equiv \inf \sum_{J\in \ganz, M \in \zn} w(Q_{J,M})^{-\left(\frac{1}{p} + \frac{\vr}{n}\right)} \| h_{J,M} w^{-\frac{1}{p'}} \, | L_{p} (\rn) \|
\end{\eq}
where the infimum is taken over all representations \eqref{2.9}, \eqref{2.10a}.
\end{Def}
\begin{Prop}   \label{T3.1}
Let $1<p<\infty$, $-n <\vr <-n/p$, $w\in A_{p'}$ with $p'=\frac{p}{p-1}$ and let $u\equiv -\frac{n}{\vr}$.
Then $1<u<p$, $w^{1-u}\in A_u$ and
\begin{\eq}  \label{3.5}
\begin{aligned}
\Hrow \hra L_{u,w^{1-u}} (\rn) 
\end{aligned}
\end{\eq}
Furthermore, $D(\rn)$ is dense in $\Hrow$.
\end{Prop}
\begin{proof}
  We observe $1<u<p$. This implies $p'<u'$ and hence $w\in A_{u'}$ as well as $w^{1-u}\in A_u$.
	Let $h\in \Hrow$ be optimally represented according to \eqref{2.9}-\eqref{2.10b}. 
  H\"older's inequality (with respect to the measure $w(x) \di x$) yields 
	\begin{align}
	\begin{split} \label{3.5b}
	  \norm{\wt{h}}{L_{u,w^{1-u}} (\rn)} &\le \sum_{J\in \ganz, M \in \zn} \norm{h_{J,M}}{L_{u,w^{1-u}}(\rn)}
		\\ &\le \sum_{J\in \ganz, M \in \zn} \norm{h_{J,M}}{L_{p,w^{1-p}}(\rn)} w(Q_{J,M})^{-\left(\frac{1}{p} + \frac{\vr}{n}\right)} 
	\end{split}\end{align}
  and, thus, \eqref{3.5} where the first inequality holds by means of the absolute convergence of $\sum_{J,M} h_{J,M}$ in $L_{u,w^{1-u}} \hra S'(\rn)$ and $\wt{h}$ denotes the representative of $h\in S'(\rn)$.
\par
	It remains to prove that $D(\rn)$ is dense in $\Hrow$. Let $h$ be given by \eqref{2.9}, \eqref{2.10a} and 
let
\begin{\eq} \label{SY}
h^L = \sum_{|J|\le L, |M| \le L} h_{J,M}, \qquad L \in \nat.
\end{\eq}
Then
\[
\| h - h^L \, | \Hrow \| \to 0 \qquad \text{if} \quad L \to \infty.
\]
Taking into account that $w^{1-p}\in A_p$ any $h_{J,M}$ with $|J|\le L$, $|M| \le L$ can be approximated in $L_{p,w^{1-p}} (Q_{J,M})$ by functions belonging to $D(Q_{J,M})$. The sum of these functions
approximates $h^L$ and also $h$. Hence $D(\rn)$ is dense in $H^{\vr} L_p (\rn)$.
\end{proof}
\begin{Rem} \label{GF3}
  It follows from \eqref{3.5b} that assumption \eqref{2.10a} ensures unconditional convergence of \eqref{2.9} in $L_{u,w^{1-u}}(\rn)$, for $\vr u =-n$, and hence in $S'(\rn)$. 
	Moreover, we mention that if one would extend the parameter range of $\vr$ to $-n<\vr \le -n/p$, then the space $H^{-n/p} L_p (w,\rn)$ would coincide with $L_{p,\wt{w}}(\rn)$ where $\wt{w}\equiv w^{1-p}$ (cf. \cite[(2.23)-(2.25)]{RoT14_2} for the unweighted case).
	Indeed, for $h\in H^{-n/p} L_p (w,\rn)$ the triangular inequality implies
	\begin{\eq}   \label{2.21}
	\| h \, | L_{p,\wt{w}} (\rn) \| \le \| h \, | H^{-n/p} L_p (w,\rn) \|.
	\end{\eq}
We prove the converse.
Let $h\in L_{p,\wt{w}} (\rn)$.
Let $Q^J =Q_{-J,0_n} = [-2^J,2^J]^n$, $J\in \nat$, be admitted cubes. In dependence on $h$ there is a monotonically increasing sequence of natural numbers $\{J_l\}^\infty_{l=0}$ such that
\[ 
h = \sum^\infty_{l=0} h_l, \quad \supp h_l \subset {Q^{J_l}}, \quad \| h_l \, | L_{p,\wt{w}} (Q^{J_l}) \| \le 2^{-l} \|h \, | L_{p,\wt{w}} (\rn) \|.
\] 
For suitable $J_l$ one may choose $h_0 = h|Q^{J_0}$ and $h_l = h|Q^{J_l} \setminus Q^{J_{l-1}}$ if $l\in \nat$. It is an admitted decomposition in the sense of
\eqref{2.9}-\eqref{2.10b} (extended to $\vr = -n/p$). This proves $h \in H^{-n/p} L_p (w,\rn)$ and the converse of \eqref{2.21} (with an additional factor 2).
\end{Rem}
\section{Duals and preduals - the weighted case} 
\subsection{Predual spaces}
The duality with respect to unweighted Morrey spaces is discussed in the scalar case in detail with complete proofs in \cite{RoT14_2} and in the vector-valued case in \cite{RS14}. Here we give complete proofs for Muckenhoupt weighted Morrey spaces following their approach.
\begin{Thm} \label{TDp1:GG}
Let $1<p<\infty$, $-\frac{n}{p}< r<0$, $r+\vr=-n$ and let $w\in A_p$. 
Then the predual space of $\Lmwe$ is $\Hrdw$. 
Moreover,
\[g\in \left(\Hrdw\right)'\] if, and only if, $g$ can be uniquely represented as 
\begin{\eq} \label{87:GG}
  g(f)=\int_{\rn} \wt{g}(x) f(x)\di x
\end{\eq}
for all $f \in D(\rn)$ where $\wt{g} \in \Lmwe$ and
\begin{\eq} \label{MWGH}
	\left\|g\left|\left( \Hrdw \right)'\right.\right\|\cong \left\|\wt{g} |\Lmwe\right\|.
\end{\eq}
\end{Thm}
\begin{proof}
Let $\wt{g}\in\Lmwe$. 
Let $f\in \Hrdw$ be optimally represented, that is, we assume that  
\begin{align}   \label{12GIS}
\begin{split}
&f= \sum_{J\in \ganz, M\in \zn} h_{J,M} \quad \text{in} \quad S'(\rn) \quad \text{with} \quad \supp h_{J,M} \subset {Q_{J,M}},
\end{split}
\end{align}
such that
\begin{\eq*}    
\sum_{J\in \ganz, M \in \zn} w(Q_{J,M})^{-\left(\frac{1}{p} + \frac{\vr}{n}\right)} \| h_{J,M} w^{-\frac{1}{p'}} \, | L_{p} (\rn) \| < 2
\norm{f}{\Hrdw}.
\end{\eq*}
We observe that by Remark \ref{GF3} the convergence in \eqref{12GIS} holds also in $L_{u,w^{1-u}}(\rn)$ for $\vr u = -n$. The same holds even for $h\equiv \sum_{J\in \ganz, M\in \zn} |h_{J,M}|$. That is that $h^N\equiv \sum_{|J|, |M|\le N} |h_{J,M}|$ tends to $h$ in $L_{u,w^{1-u}}$ for $N\rightarrow\infty$. Therefore exists a partial sum $h^{N_l}$ of $h^N$ which convergences pointwise almost everywhere to $h$ for $l\rightarrow\infty$. But we have also that $\sum_{|J|, |M|\le N_l} h_{J,M}$ converges pointwise almost everywhere to $f$ (because absolute converging sums are convergent).
Lebesgue's monotone convergence theorem, H\"older's inequality and $r+\vr+n=0$ yield the estimates
\begin{align*}
  &\int_{\rn} \left|\wt{g}(y) f(y)\right|\di y
=\int_{\rn} \left|\wt{g}(y) \lim_{l\rightarrow\infty} \sum_{|J|, |M|\le N_l} h_{J,M}(y)\right|\di y
	\\ \leq & \sum_{J\in \ganz, M\in \zn} \int_{\rn} \left|\wt{g}(y) h_{J,M}(y)\right|w^{\frac{1}{p}} w^{-\frac{1}{p}} \di y
  \\ \leq & 
	\sum_{J\in \ganz, M\in \zn}  w(Q_{J,M})^{-\left(\frac{1}{p} + \frac{r}{n}\right)} \left\| \left. \wt{g} \right| L_{p,w}(Q_{J,M}) \right\|
  \\ & \cdot w(Q_{J,M})^{-\left(\frac{1}{p'} + \frac{\vr}{n}\right)} \| h_{J,M} w^{-\frac{1}{p}} \, | L_{p'} (\rn) \|
	\\ \leq & 2 \left\|\wt{g}|\Lmwe \right\| \left\|f|\Hrdw \right\|.
\end{align*}
Hence, in particular, any $\wt{g} \in \Lmwe$ induces a bounded linear functional on $\Hrdw$. 
\par Conversely, suppose that $g$ is a bounded linear functional on $\Hrdw$ with the norm $\left\|g\right\|$.
Let $\varphi\in D(Q_{J,M})$. Then
\begin{\eq}   \label{4.31}
\| \varphi \, | \Hrdw \| \le   w(Q_{J,M})^{-\left(\frac{1}{p'} + \frac{\vr}{n}\right)} \| \varphi w^{-\frac{1}{p}} \, | L_{p'} (\rn) \|  <\infty,
\end{\eq}
where the first inequality holds by Definition \ref{D2.3} and the second inequality because of $w^{-\frac{p'}{p}}=w^{1-p'} \in A_{p'}$ and the fact that $D(\rn)$ is embedded in Muckenhoupt weighted Lebesgue spaces. According to our assumption on $g$ we obtain further
\[
  |g(\varphi)| \le \left\|g\right\| w(Q_{J,M})^{-\left(\frac{1}{p'} + \frac{\vr}{n}\right)} \| \varphi w^{-\frac{1}{p}} \, | L_{p'} (Q_{J,M}) \|
\]
and particularly $g\in (L_{p',w^{1-p'}} (Q_{J,M}))'$.
Using the duality $ (L_{p',w^{1-p'}} (Q_{J,M}))'\cong L_{p,w}(Q_{J,M})$ we deduce
\[
  \norm{g^{Q_{J,M}}}{L_{p,w}(Q_{J,M})} \le \left\|g\right\| w(Q_{J,M})^{-\left(\frac{1}{p'} + \frac{\vr}{n}\right)}.
\]
with some $g^{Q_{J,M}}\in L_{p,w}(Q_{J,M})$ such that
\[
  g(\varphi) = \int_{Q_{J,M}} g^{Q_{J,M}}(x)\varphi(x) \di x .
\]
Taking a sequence of dyadic cubes denoted by $Q^l$, $l\in\nat$, such that $Q^{l-1} \subset Q^l$ and $\bigcup_{l\in\nat} Q^l = \rn$ we get a single function $\tilde{g}$ on $\rn$ which equals the function $g^{Q_{J,M}}$ on $Q_{J,M}$ and which is in $L_{p,w}(Q_{J,M})$ for all $J\in\ganz$, $M\in\ganz^n$. Moreover, $\tilde{g}$ satisfies then
\begin{\eq} \label{repre}
  g(\varphi) = \int_{\rn} \tilde{g}(x)\varphi(x) \di x 
\end{\eq}
for $\varphi\in D(\rn)$.
This proves $\tilde{g}\in \Lmwe$ and
\[
  \norm{\tilde{g}}{\Lmwe}\le \left\|g\right\|
\]
by $\frac{n}{p'} +\vr = -\frac{n}{p} -r$. 
\end{proof}
\begin{Rem}
 An alternative proof of the last result in a more general setting can be found in \cite{ST09}.
\end{Rem}
\begin{Prop} \label{PDM}
Let $1<p<\infty$, $-\frac{n}{p}< r<0$, $r+\vr=-n$ and let $w\in A_p$. 
Then it holds
\begin{align} \begin{split} \label{58:GGG}
	\int_{\rn}  \left|g(y) f(y)\right| \di y 
	\le \left\|g|\Lmwe \right\| \left\|f|\Hrdw \right\|
\end{split} \end{align}
for $g\in \Lmwe$ and $f \in \Hrdw $. 
Moreover, holds even equality in \eqref{MWGH}.
Furthermore, for ${g}\in \Lmwe$ we have
\begin{\eq} \label{1:GHM}
  \left\|{g}|\Lmwe\right\|=\sup_{f}\left|\int_{\rn} {g}(x) f(x)\di x\right| 
\end{\eq}
where the supremum is taken over all $f \in D(\rn)$ with $\left\|f|\Hrdw\right\|\leq 1$. 
\end{Prop}
\begin{proof} 
Let $\wt{g}\in\Lmwe$, $f\in \Hrdw$ and let be $\ve>0$. 
Let $f\in \Hrdw$ be represented such that  
\begin{align*}   
\begin{split}
&f= \sum_{J\in \ganz, M\in \zn} h_{J,M} \quad \text{in} \quad S'(\rn) \quad \text{with} \quad \supp h_{J,M} \subset {Q_{J,M}},
\end{split}
\end{align*}
such that
\begin{\eq*}    
\sum_{J\in \ganz, M \in \zn} w(Q_{J,M})^{-\left(\frac{1}{p} + \frac{\vr}{n}\right)} \| h_{J,M} w^{-\frac{1}{p'}} \, | L_{p} (\rn) \| < (1+\ve)
\norm{f}{\Hrow}.
\end{\eq*}
Now we deduce the desired statements arguing as in the proof of Theorem \ref{TDp1:GG}.
\end{proof} 
\begin{Rem}
Theorem \ref{TDp1:GG} has some history in the unweighted situation (cf. \cite{Lon84, Zor86, Kal98, GR01, AX04, GM13, Ros13, RoT14_2}).
\end{Rem}
\subsection{Dual spaces}
In the proof of the next theorem, 
we benefit from the following general assertion.
\begin{Prop}[p. 73 of \cite{ET96}, Lemma in Section 1.11.1 of \cite{T78}] \label{86:GG}
Let $\{ A_j \}_{j\in\nat_0}$ be a sequence of complex Banach spaces and $\{A_j'\}_{j\in\nat_0}$ their respective duals. 
	For
   	\begin{align*}
		c_0(\{A_j\}) &\equiv \left\{  \left. a\equiv\left\{a_j\right\}_{j\in\nat_0} \right| a_j\in A_j,  \right.\\ &\left. 
		\quad\left\|a|c_0(A_j)\right\|\equiv\left\|a|\ell_\infty(A_j)\right\|\equiv\sup_j\left\|a_j|A_j\right\|<\infty, \left\|a_j|A_j\right\|\rightarrow 0\right\},
		\\
		\ell_1(\{A_j'\}) &\equiv \left\{  \left. a'\equiv\left\{a_j'\right\}_{j\in\nat_0} \right| a_j'\in A_j', \left\|a'|\ell_1(A_j')\right\|\equiv\sum_j\left\|a_j|A_j'\right\|<\infty \right\}
	\end{align*}
it holds
\begin{align*}
  &\left(c_0(\{A_j\})\right)'=\ell_1(\{A_j'\}) \text{ with } a'(a)= \sum_{j=0}^\infty a_j'(a_j) \text{ and }\\ &
	\left\|\cdot\left| (c_0(A_j))'\right. \right\| =\left\|\cdot \left| \ell_1(A_j') \right.\right\|.
\end{align*}
\end{Prop}
\begin{Lem}[\S 8.6 in \cite{Ste93}, p. 39] \label{Mine2}
Let $w\in A_p$ for $1<p<\infty$. Then $w\notin L_1(\rn)$.
\end{Lem}
\begin{proof} 
  Let $Q$ be a cube. Moreover, let us assume $w\in L_1(\rn)$. Taking into account the reverse doubling condition of $w$ (cf. \eqref{RDC}) we deduce 
	\[
	  w(Q)\le c^l w(2^l Q)\le c^l w(\rn) <\infty, \qquad l\in\nat
	\]
	for a constant $c<1$. Passing to the limit $l\rightarrow\infty$ we obtain $w(Q)=0$ for arbitrary cubes $Q$ and hence $w$ is a.e. identically zero which contradicts our assumption $w\in A_p$.
\end{proof} 
\begin{Prop}  \label{HFS}
Let $1<p<\infty$, $-n <\vr <-n/p$ and let $w\in A_{p'}$ with $p'=\frac{p}{p-1}$. 
Then the {predual Morrey space} $\Hrow$ collects all $h\in S'(\rn)$ 
which can be represented as
\begin{align}   \label{2.9p}
\begin{split}
&h= \sum_{x\in\rat^n, J\in \ganz} h_{x,J} \quad \text{in} \quad S'(\rn) \quad \text{with} \quad \supp h_{x,J} \subset {Q(x,J)},
\end{split}
\end{align}
such that
\begin{\eq}    \label{2.10p}
\sum_{x\in\rat^n, J\in \ganz} w(Q(x,J))^{-\left(\frac{1}{p} + \frac{\vr}{n}\right)} \| h_{x,J} w^{-\frac{1}{p'}} \, | L_{p} (\rn) \| <\infty,
\end{\eq}
where ${Q(x,J)}$ is defined as in \eqref{Meq2}.
Moreover, we have the equivalence of norms
\begin{align} 
\| h \, | \Hrow \|^\ast \equiv & \inf \sum_{x\in\rat^n, J\in \ganz} w(Q(x,J))^{-\left(\frac{1}{p} + \frac{\vr}{n}\right)} \| h_{x, J} w^{-\frac{1}{p'}} \, | L_{p} (\rn) \|
\notag\\ \cong & \| h \, | \Hrow \| \label{2.10bp}
\end{align}
where the infimum is taken over all representations \eqref{2.9p}, \eqref{2.10p}.
\end{Prop}
\begin{proof} 
  Let $h\in S'(\rn)$ with $\| h \, | \Hrow \|^\ast<\infty$ be optimally represented according to \eqref{2.9p}-\eqref{2.10bp}.
	We define 
	\[
	  \wt{h_{J,M}} \equiv \sum_{x\in\rat^n:\ x\in Q_{J,M}} h_{x,J}
	\]
	for all $J\in\ganz$, $M\in\ganz^n$. Then $\supp \wt{h_{J,M}} \subset 2 Q_{J,M}$. The cube $2 Q_{J,M}$ is contained in at least one cube $Q_{J-2,\wt{M}}$ with an appropriate $\wt{M}\in\ganz^n$ (cf. \eqref{WFM1}).
	Then we define for $J\in\ganz$
	\[
	  h_{J-2,\wt{M}} \equiv 
		\sum_{\substack{ M\in\ganz^n:\ 2 Q_{J,M}\subset Q_{J-2,\wt{M}} 
		\\ \text{ and } \wt{h_{J,M}} \text{ is not summed up in another } h_{J-2,\check{M}} \text{ with } \check{M}\neq \wt{M} }}
		\wt{h_{J,M}}, \qquad \wt{M} \in\ganz^n
	\]
	where we point out that the second condition with respect to the summation in the definition of $h_{J-2,\wt{M}}$ ensures that
	\[
	  \sum_{x\in\rat^n, J\in \ganz} h_{x, J} = \sum_{J\in \ganz, \wt{M} \in \zn} h_{J-2,\wt{M}}.
	\]
	Finally,
	\begin{align*}
	  &\sum_{J\in \ganz, \wt{M} \in \zn} w(Q_{J-2,\wt{M}})^{-\left(\frac{1}{p} + \frac{\vr}{n}\right)} \| h_{J-2,\wt{M}} w^{-\frac{1}{p'}} \, | L_{p} (\rn) \|
		\\ \le 
		&\sum_{x\in\rat^n, J\in \ganz} w(Q(x,J))^{-\left(\frac{1}{p} + \frac{\vr}{n}\right)} \| h_{x,J} w^{-\frac{1}{p'}} \, | L_{p} (\rn) \| \le 2 \| h \, | \Hrow \|^\ast
	\end{align*} 
	where the last but one inequality holds by the fact $Q(x,J)\subset 2 Q_{J,M}$ for $x\in Q_{J,M}$.
\end{proof} 
\begin{Thm} \label{TDp2:GG}
Let $1<p<\infty$, $-\frac{n}{p}< r<0$, $r+\vr=-n$ and let $w\in A_p$. 
Then the dual space of $\Lciw$ is $\Hrdw$. 
Moreover,
$g\in \left(\Lciw\right)'$ if, and only if, $g$ can be uniquely represented as 
\begin{\eq} \label{represDual}
  g(f)=\int_{\rn} \wt{g}(x) f(x)\di x
\end{\eq}
for all $f\in L_{-\frac{n}{r},w}(\rn)$ where $\wt{g}\in \Hrdw$ and
\begin{\eq}  \label{represDualNorm}
	\left\|g\left|\left(\Lciw\right)'\right.\right\|
	\cong \left\|\wt{g}|\Hrdw\right\|.
\end{\eq}
\end{Thm}
\begin{proof}
It follows from \eqref{58:GGG} that any $g\in \Hrdw$ induces a bounded linear functional on $\Lciw$. 
\par Conversely, suppose that $g$ is a bounded linear functional on $\Lciw$ with norm $\left\|g\right\|$. 
Let $f\in D(\rn)$ and $f_{x, J} = f \, \chi_{Q(x,J)}$ for $J\in \ganz$ and $x\in\rat^n$. 
Then
\begin{align} \begin{split}   \label{4.18}
& w(Q(x,J))^{-\left( \frac{1}{p} + \frac{r}{n}\right)} \| f_{x, J} \, | L_{p,w} (Q(x,J)) \|
\\ \le & w(Q(x,J))^{-\left( \frac{1}{p} + \frac{r}{n}\right)} \| f \, | L_{p,w} (\rn) \| \to 0 
\end{split}\end{align}
if $J \rightarrow -\infty$. 
This follows from $\frac{n}{p} +r >0$, Lemma \ref{Mine2} and Lebesgue's monotone convergence theorem.
Moreover,
\begin{\eq}   \label{4.19}
w(Q(x,J))^{-\left( \frac{1}{p} + \frac{r}{n}\right)} \| f_{x, J} \, | L_{p,w} (Q(x,J)) \| 
\le w(Q(x,J))^{-\frac{r}{n}} \| f \, | L_\infty (\rn) \| 
\to 0 
\end{\eq}
if $J \rightarrow \infty$
follows from $r<0$ and $w(Q(x,J))\to 0$ for $J \rightarrow \infty$ (which holds by Lebesgue's dominated convergence theorem and $w\in L_1^{\text{loc}}(\rn)$). Furthermore, for fixed $J_0\in \nat$ and $|J|\le J_0$ we have 
\begin{\eq}   \label{4.20}
w(Q(x,J))^{-\left( \frac{1}{p} + \frac{r}{n}\right)} \| f_{x, J} \, | L_{p,w} (Q(x,J)) \| \to 0 \quad \text{if} \quad |x| \to \infty,
\end{\eq}
since $\supp(f) \cap Q(x,-J_0)=\emptyset$ for $|x|> l$ and $l=l(f, x, J_0)\in \nat$ being sufficiently large.
Using Lemma \ref{p4.2} we observe further that
\begin{align} \begin{split} \label{4.21}
    \left\|f|\Lmwe\right\|& \cong \sup_{x\in\rat^n, J\in\ganz} \Big( \int_{Q(x,J)} |f(y)|^p \,  w(Q(x,J))^{-\left(1 + \frac{rp}{n}\right)} w(y) \di y \Big)^\frac{1}{p} \\&= \left\| F |c_0\left(L_{p,\mu_{x, J}}(Q(x,J))\right)\right\|
\end{split}\end{align}
  where 
	$F\equiv \{ f_{x, J}\}_{x\in\rat^n, J\in\ganz}$,
	$\mu_{x, J}(\di y)\equiv w(Q(x,J))^{-\left(1 + \frac{rp}{n}\right)} \, w(y) \di y$ and 
  \begin{align*}
    &\left\|F|c_0\left(L_{p,\mu_{x, J}}(Q(x,J))\right)\right\|
		\\ \equiv & 
		\sup_{x\in\rat^n, J\in\ganz} \Big( \int_{Q(x,J)} |f_{x, J}(y)|^p \,  w(Q(x,J))^{-\left(1 + \frac{rp}{n}\right)} w(y) \di y \Big)^\frac{1}{p}. 
  \end{align*}
  	Combining \eqref{4.18}-\eqref{4.21} it follows that $\Lciw$ is isomorphic to a closed subspace of $c_0\left(L_{p,\mu_{x, J}}(Q(x,J))\right)$. 
	More precisely, we have a linear, surjective map $I:f \mapsto \{f_{x, J}\}_{x, J}$ from $\Lciw$ onto the closed subspace $\{\{f_{x, J}\} | f\in \Lciw \}$ of $c_0\left(L_{p,\mu_{x, J}}(Q(x,J))\right)$ satisfying \eqref{4.21},
	\[
	  I \Lciw = \{\{f_{x, J}\}_{x\in\rat^n, J\in\ganz} | f\in \Lciw \} \hra c_0\left(L_{p,\mu_{x, J}}(Q(x,J))\right).
	\]
  Hahn-Banach's theorem yields
$g\in \left(\Lciw\right)'$ if, and only if, \\$g\in \left(c_0\left(L_{p,\mu_{x, J}}(Q(x,J))\right)\right)'$ and by Proposition \ref{86:GG} 
we have the representation
\begin{align}
\label{4.32a}
 & g(f)=\sum_{x\in\rat^n, J\in\ganz} \int_{Q(x,J)}  f(y) g_{x, J}(y) \, w(Q(x,J))^{-\left(1 + \frac{rp}{n}\right)} \di y
\end{align}
 for any $f\in \Lciw$ with $\{ g_{x, J} \}\in \ell_1\left(L_{p',\wt{\mu_{x, J}}}(Q(x,J))\right)$
where
 \begin{align*}
   &\left\|\{ g_{x, J} \}|\ell_1\left(L_{p',\wt{\mu_{x, J}}}(Q(x,J))\right)\right\|
   \\=& 
   \sum_{x\in\rat^n, J\in\ganz} \Big( \int_{Q(x,J)} |g_{x, J}(y)|^{p'}  \, w(Q(x,J))^{-\left(1 + \frac{rp}{n}\right)} w(y)^{-\frac{p'}{p}} \di y \Big)^\frac{1}{p'}
 \end{align*}
and $\wt{\mu_{x, J}}(\di y)\equiv w(Q(x,J))^{-\left(1 + \frac{rp}{n}\right)} w(y)^{-\frac{p'}{p}} \di y$.
To get \eqref{4.32a} we used 
\[
\left(L_{p,\mu_{x, J}}(Q(x,J))\right)'=L_{p',\wt{\mu_{x, J}}}(Q(x,J))  \qquad , x\in\rat^n, J\in\ganz,
\]
in order to apply Proposition \ref{86:GG} which is deduced analogously to the well-known duality assertion in Muckenhoupt weighted Lebesgue spaces $\left(L_{p,w}(Q(x,J))\right)'=L_{p', w^{1-p'}}(Q(x,J))$ (where $w\in A_p$) replacing the Lebesgue measure by the Lebesgue measure multiplied with the constant $w(Q(x,J))^{-\left(1 + \frac{rp}{n}\right)}$.   
Moreover, Hahn-Banach's theorem and Lemma \ref{p4.2} enable us to assume 
\begin{\eq} \label{as8}
  \left\|g\left|\left(\Lci\right)'\right.\right\|\cong\left\|\{ g_{x, J} \}\left|\ell_1\left(L_{p',\wt{\mu_{x, J}}}(Q(x,J))\right)\right.\right\|.
\end{\eq}
Using Lebesgue's dominated convergence theorem (cf. \eqref{Lconv} for an integrable majorant) 
we deduce from \eqref{4.32a} the representation
\begin{align}
\label{as9}
 & g(f)
	= \int_{\rn} f(y) \wt{g}(y) \di y 
\end{align}
 for any $f\in D(\rn)$ and 
\[
  \wt{g}(y) \equiv \sum_{x\in\rat^n, J\in\ganz} g_{x, J}(y) \chi_{Q(x,J)}(y) \, w(Q(x,J))^{-\left(1 + \frac{rp}{n}\right)}.
\]
With
\[
h_{x,J} = g_{x, J}(y) \chi_{Q(x,J)}(y) \, w(Q(x,J))^{-\left(1 + \frac{rp}{n}\right)}
\]
and $r+\vr = -n$ one has
\begin{align*}
&w(Q(x,J))^{-\left(\frac{1}{p'} + \frac{\vr}{n}\right)} \Big( \int_{Q(x,J)} |h_{x,J} (y) |^{p'} w(y)^{-\frac{p'}{p}} \, \di y\Big)^{1/p'} \\
=&  w(Q(x,J))^{\frac{1}{p} + \frac{r}{n}} \left( \int_{Q(x,J)} |g_{x, J}(y)|^{p'} \, w(Q(x,J))^{-\left(1 + \frac{rp}{n}\right)(p' -1)} 
\right.\\ &\left.
\qquad\qquad\qquad\qquad\cdot\quad w(Q(x,J))^{-\left(1 + \frac{rp}{n}\right)} w(y)^{-\frac{p'}{p}} \, \di y \right)^{1/p'} \\
=& \Big( \int_{Q(x,J)} |g_{x, J}(y)|^{p'} \, w(Q(x,J))^{-\left(1 + \frac{rp}{n}\right)} w(y)^{-\frac{p'}{p}} \, \di y\Big)^{1/p'} .
\end{align*}
Using Proposition \ref{HFS} and \eqref{as8} we see therefore that
\begin{\eq*}  
  \|\wt{g} \, | \Hrdw \| \le c_1 \left\|\{ g_{x, J} \}\left|\ell_1\left(L_{p',\wt{\mu_{x, J}}}(Q(x,J))\right)\right.\right\| \le c_2 \left\|g\right\|
\end{\eq*}
where $\wt{g} = \sum_{x\in\rat^n, J\in\ganz} h_{x,J}$ converges in $L_{u,w^{1-u}}(\rn)$ with $\vr u = -n$ by Proposition \ref{T3.1} and hence in $S' (\rn)$ (noting that $1<u<p'$ and $w^{1-u}\in A_u$) . 
We can establish an integrable majorant to justify \eqref{as9} (resp. the application of Lebesgue's dominated convergence theorem) by the same argumentation using $f\in D(\rn)$. Indeed, 
\begin{\eq}  \label{Lconv}
  |f|\sum_{x\in\rat^n, J\in\ganz} |h_{x,J}|=\sum_{x\in\rat^n, J\in\ganz} |f\, h_{x,J}|
\end{\eq}	
is integrable by H\"older's inequality since $\sum_{x\in\rat^n, J\in\ganz} |h_{x,J}| \in \Hrdw \hra L_{u,w^{1-u}} (\rn)$, $\vr u = -n$, analogously to $\wt{g}\in \Hrdw$ and \eqref{3.5} and moreover $f\in L_{u',w} (\rn)$ where $w\in A_{u'}$ by $p<u'$.
It remains to show that $\wt{g}$ represents $g$ for all $f\in L_{-\frac{n}{r},w}(\rn)$ cf. \eqref{represDual} and not only for all $f\in D(\rn)$. Let $f\in L_{-\frac{n}{r},w}(\rn)$. Then there is a sequence $f_n \in D(\rn)$, $n\in\nat$, with $f_n \rightarrow f$ almost everywhere and $f_n \rightarrow f$ in $L_{-\frac{n}{r},w}(\rn)$ and moreover such that $f$ can be dominated by $f_n$, i.e. $|f_n|\le |f|$. Because of the continuity of $g$, \eqref{represDual} for $f_n \in D(\rn)$ and Lebesgue's dominated convergence theorem ($|\wt{g}(\cdot)| f(\cdot)$ is an integrable majorant by \eqref{58:GGG}) we obtain 
\begin{align*}
  g(f)&=\lim_{n\in\nat} g(f_n)= \lim_{n\in\nat} \int_{\rn} \wt{g}(x)f_n(x)\di x = \int_{\rn} \wt{g}(x) \lim_{n\in\nat}  f_n(x)\di x \\&= \int_{\rn} \wt{g}(x) f(x)\di x
\end{align*}
and hence \eqref{represDual} for all $f\in L_{-\frac{n}{r},w}(\rn)$.
\end{proof}
\begin{Prop} \label{DM}
Let $1<p<\infty$, $-\frac{n}{p}< r<0$, $r+\vr=-n$ and let $w\in A_p$. 
Using the norms 
\begin{align}
  \norm{f}{\Lmwe}^\ast 
	 \equiv &  \sup_{x\in\rat^n, J\in\ganz} w(Q(x,J))^{-\left(\frac{1}{p}+\frac{r}{n} \right)} \left\|f|L_{p, w}(Q(x,J))\right\| 
\end{align}
(cf. \eqref{Meq2}) and $\norm{\cdot}{\Hrdw}^\ast$ (cf. \eqref{2.10p}) we obtain even equality in \eqref{represDualNorm}. That is,
$g\in \left(\Lciw\right)'$ if, and only if, $g$ can be uniquely represented as 
\[ 
  g(f)=\int_{\rn} \wt{g}(x) f(x)\di x
\] 
for all $f\in L_{-\frac{n}{r},w }(\rn)$ where $\wt{g}\in \Hrdw$ and
\[
	\left\|g\left|\left(\Lciw\right)'\right.\right\|^\ast\equiv 
	\sup_{\substack{f\in \Lciw:\\\norm{f}{\Lmwe}^\ast\le 1}} |g(f)| =\left\|\wt{g}|\Hrdw\right\|^\ast.
\]
Moreover, for $\wt{g}\in \Hrdw$ it holds 
\begin{\eq} \label{1GIS}
  \left\|\wt{g}|\Hrdw\right\|^\ast=\sup_{f}\left|\int_{\rn} \wt{g}(x) f(x)\di x\right|
\end{\eq}
where the supremum is taken over all $f\in L_{-\frac{n}{r},w }(\rn)$ with $\|f | \Lmwe\|^\ast\leq 1$. 
\end{Prop}
\begin{proof}
Analogously to the arguments establishing \eqref{58:GGG} we get 
\[
    \left|\int_{\rn}  f(y)\wt{g}(y) \di y\right| 
	\le \left\|f|\Lmwe \right\|^\ast \left\|\wt{g}|\Hrdw \right\|^\ast 
\]
for $f\in \Lciw$ and $\wt{g}\in \Hrdw$. This yields for the induced functional $g$ of $\wt{g}$
\[
  \left\|g\left|\left(\Lciw\right)'\right.\right\|^\ast\le \left\|\wt{g}|\Hrdw\right\|^\ast.
\]
\par
Conversely, suppose $g$ is a bounded linear functional on $\Lciw$ with norm $\left\|g\left|\left(\Lciw\right)'\right.\right\|$. Then we show by the same arguments as in the proof of Theorem \ref{TDp2:GG} that $g$ can be represented as in \eqref{represDual} where $\wt{g}$ also satisfies
\[
  \left\|\wt{g}|\Hrdw \right\|^\ast \le \left\|g\left|\left(\Lciw\right)'\right.\right\|^\ast.
\]
\end{proof}
\begin{Rem}
The unweighted case of Theorem \ref{TDp2:GG} is stated in \cite{AX12} and proved in \cite{Ros13,RoT14_2}. In dimension $n=1$ it is also proved in a periodic setting on the torus in \cite{ISY14}. Its vector-valued version is investigated in \cite{RS14}.
\end{Rem}
\begin{Corol}   
\begin{enumerate}[label=(\roman{*}), ref=(\roman{*})]
\item
Let $1<p<\infty$, $-\frac{n}{p}< r<0$ and let $w\in A_p$.
Then $\Lciw$ and $\Lmwe$ are Banach spaces.
\item
Let $1<p<\infty$, $-n <\vr <-n/p$ and let $w\in A_{p'}$ with $p'=\frac{p}{p-1}$.
Then $\Hrow$ are Banach spaces.
\end{enumerate}
\end{Corol}
\begin{proof}
  The spaces $\Lmwe$ and $\Hrow$ are duals of normed vector spaces ($\Hrow$ resp. $\Lciw$) and therefore complete. Since $\Lciw$ is a closed subspace of the complete space $\Lmwe$ we get also the completeness of $\Lciw$.
\end{proof}
\begin{Corol}   \label{11GIS}
Let $1<p<\infty$, $-n <\vr <-\frac{n}{p}$ and let $w\in A_{p'}$ with $p'=\frac{p}{p-1}$.
\begin{enumerate}[label=(\roman{*}), ref=(\roman{*})]
\item
If $|f|\le |g|$ almost everywhere, then $\norm{f}{\Hrow} \le c \norm{g}{\Hrow}$ where the constant $c$ is independent of $f$ and $g$.  \label{10GIS}
\item Moreover, it holds $\norm{f}{\Hrow}\cong \norm{\: |f| \; }{\Hrow}$. \label{9GIS}
\end{enumerate}
\end{Corol}   
\begin{proof}
Let $|f|\le |g|$ almost everywhere. 
By \eqref{1GIS} we have
\begin{align*}
  &\left\|f|\Hrow \right\|^\ast \le \sup_{h} \int_{\rn} |f(x)| |h(x)|\di x 
	\le \sup_{h} \int_{\rn} |f(x)| h(x)\di x
	\\ \le & \sup_{h} \int_{\rn} |g(x)| h(x)\di x
	= \left\|\: |g| \; |\Hrow \right\|^\ast
\end{align*}
where the supremum is taken over all $h\in L_{-\frac{n}{r},w }(\rn)$ with $\|h | \Lmwe\|^\ast\leq 1$ and which implies
\begin{\eq} \label{8GIS}
  \left\|f|\Hrow \right\| \le c_1 \left\|\: |f| \;|\Hrow \right\| \le c_2 \left\|\: |g| \;|\Hrow \right\|
\end{\eq}
Let $f\in \Hrow$ be optimally represented, that is, we assume that  
\begin{align*}   
\begin{split}
&f= \sum_{J\in \ganz, M\in \zn} h_{J,M} \quad \text{in} \quad S'(\rn) \quad \text{with} \quad \supp h_{J,M} \subset {Q_{J,M}},
\end{split}
\end{align*}
such that
\begin{\eq*}    
\sum_{J\in \ganz, M \in \zn} w(Q_{J,M})^{-\left(\frac{1}{p} + \frac{\vr}{n}\right)} \| h_{J,M} w^{-\frac{1}{p'}} \, | L_{p} (\rn) \| < 2
\norm{f}{\Hrow}.
\end{\eq*}
Thus,
\[
  |f|\le \sum_{J\in \ganz, M\in \zn} |h_{J,M}| \equiv \overline{g}
\]
and 
\begin{align*}
  \norm{\overline{g}}{\Hrow}\le & \sum_{J\in \ganz, M \in \zn} w(Q_{J,M})^{-\left(\frac{1}{p} + \frac{\vr}{n}\right)} \| h_{J,M} w^{-\frac{1}{p'}} \, | L_{p} (\rn) \| 
	\\ < & 2 \norm{f}{\Hrow}.
\end{align*}
Using \eqref{8GIS} it follows
\[
  \norm{\: |f| \;}{\Hrow} \le c \norm{f}{\Hrow}
\]
and therefore \ref{9GIS}. Taking again into account \eqref{8GIS} we obtain finally \ref{10GIS}.
\end{proof}

\section{The extrapolation to Morrey spaces}
\begin{Def} 
The \textit{Hardy-Littlewood maximal operator} $M$ is given by 
	\begin{equation*} 
		(M f)(y)\equiv\sup_{Q:y\in Q} \frac{1}{|Q|} \int_{Q}\left|f(z)\right|\di z, \qquad f \in L_1^\text{loc}(\rn) 
	\end{equation*}
	where the supremum is taken over all cubes $Q$ (whose sides are parallel to the coordinate axes) in $\rn$ which contain $y\in \rn$.
\end{Def}
The following proposition establishes the boundedness of the maximal operator in the predual Morrey spaces. We will use this fact for the extrapolation result appearing in Theorem \ref{T3.3} below. Only the boundedness of the Hardy-Littlewood maximal operator (in the predual Morrey spaces) is required for its proof beside the duality. By the fact that the same arguments even give boundedness results of some singular integrals (in the predual Morrey spaces) we formulate the result in a more general version.
\begin{Prop} \label{4GIS}
Let $1<p<\infty$, $-n <\vr <-\frac{n}{p}$ and let $w\in A_{p'}$ with $p'=\frac{p}{p-1}$. 
Let $T$ be an operator such that it holds
		 \begin{\eq} \label{3GIS} 
		T:L_{p,w^{1-p}}(\rn) \hra L_{p,w^{1-p}}(\rn).
	\end{\eq}
Moreover, $T$ satisfies the representation formula
\begin{\eq} \label{2GIS}
  |(Tf)(y)|\leq c \int_{\real^n} \frac{|f(x)|}{|y-x|^n} \di x \quad 
\end{\eq}
for all $f\in L_{p,w^{1-p}}(\rn)$ compactly supported where $y\notin \supp(f)$.
Furthermore, we assume that the operator $T$ is
\begin{enumerate}
	\item either linear \label{ToYaG}
	\item or satisfy	 
		\begin{align} \begin{split} \label{ToYG}
	  &(T (f_1 +f_2))(y) \le (T f_1)(y)+(T f_2)(y),\\ &(T f)(y) = (T(-f))(y)   \end{split}
	  	\end{align}
	  	for $f$, $f_1$, $f_2 \in L_{p,w^{1-p}}(\rn)$ and almost all $ y\in\rn$.
	\end{enumerate}
Then there is an unique continuous and bounded extension of $\wt{T}$ of $T$ to $\Hrow$ such that 
\[
  \wt{T}: \Hrow\hra\Hrow
\] 
for every $-n <\vr <-n/p$. 
\end{Prop}
\begin{proof}
Let $f\in D(\rn) \hra \Hrow$ be optimally represented, that is, we assume that  
\begin{align}   
\begin{split} \label{TTTL}
&f= \sum_{J\in \ganz, M\in \zn} h_{J,M} \quad \text{in} \quad S'(\rn) \quad \text{with} \quad \supp h_{J,M} \subset {Q_{J,M}},
\end{split}
\end{align}
such that
\begin{\eq*}    
\sum_{J\in \ganz, M \in \zn} w(Q_{J,M})^{-\left(\frac{1}{p} + \frac{\vr}{n}\right)} \| h_{J,M} w^{-\frac{1}{p'}} \, | L_{p} (\rn) \| < 2
\norm{f}{\Hrow}.
\end{\eq*}
Without loss of generality we can assume that the convergence in \eqref{TTTL} is even pointwise almost everywhere as explained at the beginning of the proof of Theorem \ref{TDp1:GG} (and as it holds for a partial sum of it).
Thus, using \eqref{ToYG} we observe that 
\begin{align} \label{7GIS} \begin{split} 
  |T(f)| &\le \sum_{J\in \ganz, M\in \zn} |T(h_{J,M})| \left(\chi_{2Q_{J,M}}+ \sum_{l\in\nat} \chi_{2^{l+1}Q_{J,M}\setminus 2^{l} Q_{J,M}} \right) \\
	 &\le \sum_{J\in \ganz, M\in \zn} \chi_{2Q_{J,M}}|T(h_{J,M})| + \sum_{J\in \ganz, M\in \zn} \sum_{l\in\nat} \chi_{2^{l+1}Q_{J,M}\setminus 2^{l} Q_{J,M}} |T(h_{J,M})|
\end{split} \end{align}
We obtain by \eqref{3GIS} and \eqref{DC}
\begin{align} \label{5GIS} \begin{split} 
  &\norm{\sum_{J\in \ganz, M\in \zn} \chi_{2Q_{J,M}}|T(h_{J,M})|}{\Hrow}
	\\ \le & \sum_{J\in \ganz, M \in \zn} w(2Q_{J,M})^{-\left(\frac{1}{p} + \frac{\vr}{n}\right)} \| \chi_{2Q_{J,M}}|T(h_{J,M})| w^{-\frac{1}{p'}} \, | L_{p} (\rn) \|
\\ \le & c_1 \sum_{J\in \ganz, M \in \zn} w(Q_{J,M})^{-\left(\frac{1}{p} + \frac{\vr}{n}\right)} \| |h_{J,M}| w^{-\frac{1}{p'}} \, | L_{p} (\rn) \| 
\\ \le & c_2 \norm{f}{\Hrow}.
\end{split}\end{align}
Furthermore,
\eqref{2GIS} and H\"older's inequality yield
\begin{align*}
  &\chi_{2^{l+1}Q_{J,M}\setminus 2^{l} Q_{J,M}} |T(h_{J,M})|
	\le \frac{c}{|2^{l}Q_{J,M}|} \int_{Q_{J,M}} |h_{J,M}(x)|\di x
	\\\le& \frac{c}{|2^{l}Q_{J,M}|} \norm{|h_{J,M}| w^{-\frac{1}{p'}}}{L_{p} (\rn)} w(Q_{J,M})^{\frac{1}{p'}}
\end{align*}
Using $w\in A_{p'}$, \eqref{RDC} and $1+\frac{\vr}{n}>0$ this implies
\begin{align} \label{6GIS} \begin{split} 
  &\norm{\sum_{J\in \ganz, M\in \zn} \sum_{l\in\nat} \chi_{2^{l+1}Q_{J,M}\setminus 2^{l} Q_{J,M}} |T(h_{J,M})|}{\Hrow}
	\\ \le & \sum_{J\in \ganz, M\in \zn} \sum_{l\in\nat} w(2^{l+1}Q_{J,M})^{-\left(\frac{1}{p} + \frac{\vr}{n}\right)} \| \chi_{2^{l+1}Q_{J,M}\setminus 2^{l} Q_{J,M}} |T(h_{J,M})| w^{-\frac{1}{p'}} \, | L_{p} (\rn) \|
\\ \le & c_1 \sum_{J\in \ganz, M \in \zn} \sum_{l\in\nat} w(2^{l+1}Q_{J,M})^{-\left(\frac{1}{p} + \frac{\vr}{n}\right)} 
\\ & \qquad \qquad \cdot 
 \frac{w(Q_{J,M})^{\frac{1}{p'}}}{|2^{l} Q_{J,M}|} \norm{|h_{J,M}| w^{-\frac{1}{p'}}}{L_{p} (\rn)} w^{1-p}(2^{l+1}Q_{J,M})^{\frac{1}{p}}
\\ \le & c_2 \sum_{J\in \ganz, M \in \zn} \left(\sum_{l\in\nat} \frac{w(Q_{J,M})^{\frac{1}{p'}+ \frac{1}{p} + \frac{\vr}{n}} }{w(2^{l+1}Q_{J,M})^{\frac{1}{p'} + \frac{1}{p} + \frac{\vr}{n}}}
\right) w(Q_{J,M})^{-\left(\frac{1}{p} + \frac{\vr}{n}\right)} \norm{|h_{J,M}| w^{-\frac{1}{p'}}}{L_{p} (\rn)} 
\\ \le & c_3 \norm{f}{\Hrow}.
\end{split}\end{align}
Finally, the assertion holds by \eqref{7GIS}, Corollary \ref{11GIS}, \eqref{5GIS}, \eqref{6GIS} and continuous and bounded extension taking into account $D(\rn)\hra \Hrow$ dense. Here we mention that whenever $T$ is not linear the continuous and bounded extension is derived as in \eqref{FGM} below. 
\end{proof}
\begin{Thm} \label{T3.3}
Assume that for some family ${\mathcal F}$ of ordered pairs of non-negative locally integrable functions $(g,f)$
, for some $1< p_1<\infty$ 
and every $w\in A_{p_1}$ we have
\begin{\eq} \label{1exa}
  \norm{g}{\Lpzwt}\le c_1 \norm{f}{\Lpzwt} \qquad \text{ for all } (g,f)\in {\mathcal F}.
\end{\eq}
Then for every
$1< p<\infty$, every $-\frac{n}{p} \le r<0$ and every $w\in A_{p}$ we have
\begin{\eq} \label{2exa}
  \norm{g}{\Lmwe}\le c_2 \norm{f}{\Lmwe} 
\end{\eq}
for all $(g,f)\in {\mathcal F}$. 
The constants $c_1$ and $c_2$ in \eqref{1exa} and \eqref{2exa} 
do not depend on $(f,g)$ but may depend on $w$, $p_1$ and $p$.
\end{Thm}
\begin{proof}
If $w\in A_{p}$, then there exists a $p_0\in (1,p)$ such that even $w\in A_{p/p_0}$.
Therefore the space $\LmwztN$ is well-defined and the Hardy-Littlewood maximal operator is bounded on its predual $\LmwztNP$ where $\vr\equiv -n-r p_0$ by Proposition \ref{4GIS}.
Fix $(g,f)\in {\mathcal F}$. By the duality representation \eqref{1:GHM} we have
\begin{align*}
  \norm{g}{\Lmwe}^{p_0}= \norm{g^{p_0}}{\LmwztN}
	= &\sup_{h}\left|\int_{\rn} {g^{p_0}}(x) h(x)\di x\right| 
	\\ \le & \sup_{h}\int_{\rn} {g^{p_0}}(x) |h(x)|\di x
\end{align*}
where the supremum is taken over all $h \in D(\rn)$ with $\left\|h|\LmwztNP\right\|\leq 1$. 
Therefore, it will suffice to fix such a function $h$ and show that
\[
  \int_{\rn} {g^{p_0}}(x) |h(x)|\di x \le c \norm{f}{\Lmwe}^{p_0}
\]
for a constant $c$ which does not depend on $h$.
Hence, we apply the Rubio de Francia algorithm defining 
\[
  (R|h|)(x)=\sum_{k=0}^\infty \frac{(M^k |h|) (x) }{\left(2\norm{M}{\LmwztNP\rightarrow\LmwztNP}\right)^k}
\]
using the 
predefinition that $M^0$ stands for the identity. 
Then we have 
\begin{enumerate}[label=(\roman{*}), ref=(\roman{*})]
	\item $|h(x)|\le (R|h|)(x)$ almost everywhere, \label{CoW}  
	\item $\norm{(R|h|)(\cdot)}{\LmwztNP}\le 2 \norm{|h|}{\LmwztNP}$ and \label{GiG1}
	\item $(M(R|h|))(x) \le 2 \norm{M}{\LmwztNP\rightarrow\LmwztNP} (R|h|)(x)$ almost everywhere. \label{GiG2} 
\end{enumerate}
For it we assume $f\in \Lmwe$ (otherwise \eqref{2exa} is trivial).
By \ref{CoW} we have that
\[
  \int_{\rn} {g^{p_0}}(x) |h(x)|\di x \le \int_{\rn} {g^{p_0}}(x) R(|h|)(x)\di x.
\]
By reason of \ref{GiG2}, $A_1\subset A_{p_0}$ and the extrapolation in the usual Lebesgue weighted spaces using \eqref{1exa} (cf. \cite{D11, CMP11, D13})) we obtain
\[
  \int_{\rn} {g^{p_0}}(x) R(|h|)(x)\di x \le c \int_{\rn} {f^{p_0}}(x) R(|h|)(x)\di x
\]
where the constant $c$ does not depend on $(g,f)$ (but may depend on $p_1$, $p_0$ and $w$).
By the generalized H\"older's inequality \eqref{58:GGG} and \ref{GiG1},
\begin{align*}
  \int_{\rn} {f^{p_0}}(x) R(|h|)(x)\di x \le & \norm{f^{p_0}}{\LmwztN} \norm{R(|h|)}{\LmwztNP}
	\\ \le & 2 \norm{f}{\Lmwe}^{p_0} \left\||h|\left|\LmwztNP\right.\right\|
	\\ \le & 2 c \norm{f}{\Lmwe}^{p_0} \norm{h}{\LmwztNP}
\end{align*}
where the last inequality follows by Corollary \ref{11GIS}. 
\end{proof}
\begin{Rem}
The proof adapts the results of \cite[Thm. 4.6]{CMP11} and \cite{CGCMP06} for Morrey spaces (which are not Banach function spaces) and shows that their approach for Banach function spaces using associated spaces could be used also with the dual or predual and a density relation instead of associated spaces for function spaces which are not Banach function spaces.
The extrapolation result \eqref{i1}$\Rightarrow$\eqref{i2} is stated here in pairs of functions free of operators. This setting has several advantages. For example, it allows to derive vector-valued inequalities without any further effort (cf. \cite{CMP04, CMP11, D13}). We proved \eqref{i3} in this general version. 
\par
The unweighted situation of Proposition \ref{4GIS} is proved in \cite{IST15}. We want to mention that if one wants to derive the boundedness of some singular integral using Proposition \ref{4GIS}, e.g. the Hilbert transform, one needs for satisfying \eqref{2GIS} that already its maximally truncations are studied in the appropriate weighted Lebesgue space (cf. \eqref{3GIS}). 
\par However, beside its preduals the maximal operator is also bounded in the weighted Morrey spaces as introduced in Definition \ref{d1:def} (cf. \cite[Thm. 3.2]{KS09}, \cite[Rem. 3]{Mu12}, \cite[Thm. 1.6.3]{Ros13} and \cite[Cor. 4.2]{KGS14}). For modifications of weighted Morrey spaces we refer to \cite{Sam14, Sam13, RSS, Saw05} and the references given there.
\end{Rem}
\section{Mapping properties of operators}
\subsection{The main theorem}
\begin{Thm} \label{vB:thm}
	Let $T$ be an operator and suppose that for some $p_0\in (1,\infty)$ and for every $w\in A_{p_0}$ it holds
		 \begin{\eq} \label{vb2c:GGl} 
		T:L_{p_0,w}(\rn) \hra L_{p_0,w}(\rn).
	\end{\eq}
	Moreover, assume that the operator $T$ is
\begin{enumerate}
	\item either linear \label{ToYa}
	\item or satisfy	 
		\begin{align} \begin{split} \label{ToY}
	  &(T (f_1 +f_2))(y) \le (T f_1)(y)+(T f_2)(y),\\ &(T f)(y) = (T(-f))(y)   \end{split}
	  	\end{align}
	  	for $f$, $f_1$, $f_2 \in D(\rn)$ and almost all $ y\in\rn$.
	\end{enumerate}
Then, the following statements hold true.
\begin{enumerate} 
\item \label{60:GG}
There are unique continuous and bounded extensions $\wt{T}$ of $T$ to $\Lciw$ such that 
\[
  \wt{T}: \Lciw\hra\Lciw
\] 
for every $1<p<\infty$, every $-\frac{n}{p}< r<0$ and every $w\in A_p$.
\item \label{59:GG}
If $T$ is furthermore linear, then there are linear and bounded extensions $\widetilde{T}$ of $T$ to $\Lmwe$ such that 
\[ 
\wt{T}: \Lmwe\hra \Lmwe 
\] 
for every $1<p<\infty$, every $-\frac{n}{p}< r<0$ and every $w\in A_p$.
\item \label{60b:GG}
Let $1<p'<\infty$, $-n<\vr<-\frac{n}{p'}$ and let $w\in A_p$ with $p=\frac{p'}{p'-1}$.
If $T$ is furthermore linear, then the dual operator of the unique linear and bounded extension $\widetilde{T}$ of $T$ acting in $\Lciw$ satisfies \[\widetilde{T}': \Hrdw\hra\Hrdw .\] 
If $T$ is in addition formally self-adjoint with respect to \eqref{vb2c:GGl}, i.e. if we assume that
\begin{\eq} \label{selfadj} 
  \left\langle T f,g\right\rangle_{\left(L_{p_0,w},L_{p_0',w'}\right)}
  =\left\langle f,T g\right\rangle_{\left(L_{p_0,w},L_{p_0',w'}\right)} 
\quad\text{for all } f,g\in D(\rn) 
\end{\eq}
where $w'\equiv w^{1-p_0'}$,
then $\widetilde{T}'$ is the unique linear and bounded extension of $T$ acting in $\Hrdw$.
\end{enumerate}
\end{Thm}
\begin{proof}
Using \eqref{vb2c:GGl} Theorem \ref{T3.3} with ${\mathcal F}\equiv \{ (|Tf|,|f|) \ \left| \ f\in D(\rn) \right. \}$ yields 
\begin{\eq} \label{extrapola}
  \norm{Tf}{\Lmwe}\le c \norm{f}{\Lmwe}
\end{\eq}
for all $f\in D(\rn)$, where the constant $c$ does not depend on $f$. Hence, 
$T: D(\rn) \rightarrow \Lmwe$. 
Theorem \ref{T3.3} as well as well-known extrapolation techniques cf. \cite[Thm. 1.4]{CMP11} imply
\begin{\eq} \label{extrapol}
  \norm{Tf}{L_{p,w}(\rn)}\le c \norm{f}{L_{p,w}(\rn)}
\end{\eq}
for all $1<p<\infty$ where the constant $c$ does not depend on $f\in D(\rn)$ and therefore $T: D(\rn) \rightarrow L_{-\frac{n}{r},w}(\rn)$. Because of the embedding $L_{-\frac{n}{r},w}(\rn) \hra \Lmwe$ and the density of $D(\rn)$ in $L_{-\frac{n}{r},w}(\rn)$ we even have $T: D(\rn) \rightarrow \Lciw$. 
Indeed, let $f\in D(\rn)$. Then $Tf\in L_{-\frac{n}{r},w}(\rn)$ by \eqref{extrapol} und thus there is a sequence of functions of $D(\rn)$ which tends to $Tf$ in $L_{-\frac{n}{r},w}(\rn)$ and hence in $\Lmwe$ which shows $Tf \in \Lciw$. 
 Thus, if $T$ is linear, then it holds for the unique (linear) continuous and bounded extension $\wt{T}$ of $T$ 
\begin{\eq} \label{ToYb}
  \wt{T}: \Lciw\hra \Lciw.
\end{\eq}
If $T$ (is not linear but) satisfies \eqref{ToY}, then we have 
	\[
	  |(T f_1)(y)-(T f_2)(y)|\le (T (f_1-f_2))(y) 
	\]
	for all $f_1$, $f_2\in D(\rn)$ and almost all $y\in\rn$ and hence in combination with \eqref{extrapola} 
	\begin{align} \begin{split}\label{FGM}
	  \left\|T f_1-T f_2|\Lmwe\right\|&\le \left\|T (f_1-f_2)|\Lmwe\right\| \\&\le 
		 c \left\|f_1-f_2|\Lmwe\right\|.
	\end{split}
	\end{align}
	Therefore, $T$ is (Lipschitz-)continuous and we get the unique continuous and bounded extension $\wt{T}:\Lciw\hra\Lciw$ of $T$ using \eqref{FGM} (and the completeness of $\Lciw$) in the same way as in the linear case.
Whenever $T$ is linear, duality directly yields the assertions
\[ 
\wt{T}'': \Lmwe\hra \Lmwe \quad \text{  and  } \quad \widetilde{T}': \Hrdw\hra\Hrdw.
\] 
Whenever $T$ satisfies \eqref{selfadj}, then we obtain by \eqref{represDual} and $T f \in L_{-\frac{n}{r},w}(\rn)$ the identities
\begin{align} \begin{split}\label{GHJ}
&\left\langle f, \widetilde{T}'g\right\rangle_{\left(\Lciw,\Hrdw\right)}=
\left\langle \widetilde{T} f,g\right\rangle_{\left(\Lciw,\Hrdw\right)}
\\=&
\left\langle T f,g\right\rangle_{\left(\Lciw,\Hrdw\right)}
\stackrel{\eqref{represDual}}{=}\int_{\rn} (T f)(x)g(x) \di x
\\=& \left\langle T f,g\right\rangle_{\left(L_{p_0,w},L_{p_0',w'}\right)}
  \stackrel{\eqref{selfadj}}{=}\left\langle f,T g\right\rangle_{\left(L_{p_0,w},L_{p_0',w'}\right)} 
\end{split}
\end{align}
for all $f,g\in D(\rn)$ where $w'= w^{1-p_0'}$. Therefore, $\widetilde{T}'g= T g$ almost everywhere for all $g\in D(\rn)$ by the fundamental lemma of the calculus of variations. Hence, $\widetilde{T}'$ is a linear and continuous extension of $T$ to $\Hrdw$. 
Moreover, the bidual $\widetilde{T}'' = \left(\widetilde{T}'\right)'$ is an extension of $T$ to $\Lr$. Indeed, we have 
\begin{align*}
&\left\langle g, \widetilde{T}''f\right\rangle_{\left(\Hrdw, \Lmwe\right)}
=\left\langle\widetilde{T}'g,f\right\rangle_{\left(\Hrdw, \Lmwe\right)}
\\=& \left\langle \widetilde{T} 'g,f\right\rangle_{\left(L_{u',\widehat{w}}(\rn),L_{u,w}(\rn)\right)}
=\int_{\rn} \widetilde{T} 'g(x) f(x) \di x
\\ \stackrel{\eqref{represDual}}{=}& \left\langle f, \widetilde{T}'g\right\rangle_{(\Lciw,\Hrdw)}
=\left\langle \widetilde{T} f,g\right\rangle_{\left(\Lciw,\Hrdw\right)}
\\ \stackrel{\eqref{GHJ}}{=} & \left\langle T f,g\right\rangle_{\left(L_{p_0,w},L_{p_0',w'}\right)}
\end{align*} 
for all $f,g\in D(\rn)$ where $u\equiv -\frac{n}{r}$ and $\widehat{w}\equiv w^{1-u'}$. 
Hereby, the second equality holds by Hahn-Banach's theorem ($f \in \Lmwe$ if, and only if, $f \in \Hrdw'$ if, and only if, $f \in L_{u',\widehat{w}}(\rn)'$ if, and only if, $f \in L_{u,w}(\rn)$). 
Therefore, $\widetilde{T}'' = T$ on $D(\rn)$ using \eqref{87:GG}. 
\end{proof}
\begin{Rem}
Theorem \ref{T3.3} and Proposition \ref{p2.5} imply the inequality
\begin{\eq} \label{ZHH}
  \norm{Tf}{\Lmwe}\le c \norm{f}{\Lmwe}
\end{\eq}
and $Tf\in \Lciw$ for all $f\in D(\rn)$ where the constant $c$ does not depend on $f$ without any further assumption on $T$ except \eqref{vb2c:GGl}.
  The assumptions that $T$ is linear or that $T$ satisfies \eqref{ToY} are just used to obtain \eqref{ToYb} from \eqref{ZHH}.
\end{Rem}
\subsection{Distinguished examples}
\subsubsection{Calder\'{o}n-Zygmund operators and its maximal truncations}
\begin{Corol} \label{CvB2:thm}
Let $1<p<\infty$, $-\frac{n}{p}< r<0$, $-n < \vr < -\frac{n}{p'}$ and let $w\in A_p$. Let $T$ be an operator with domain $D(\rn)$ satisfying
  \[
		\left\|Tf|L_{2}(\rn)\right\|	\leq c_1 \left\|f|L_{2}(\rn)\right\|
	\] 
	where the constant $c_1$ is independent of $f\in D(\rn)$ and suppose that
	\[ 
	  (T f)(y) = \lim_{\varepsilon \searrow 0} \int_{x\in\real^n, |y-x|\ge \varepsilon} K(y,x) f(x)\di x  
	\] 
	almost everywhere for all $ f\in D(\rn)$, where the function $K(\cdot,\cdot)$ defined $\rn \times \rn \setminus \{(x,x):x\in\rn\} $ satisfies the conditions $|K(x,y)|\leq c_2 |x-y|^{-n}$ and 
\begin{align*}
  &|K(x,y)-K(x',y)|\leq c_2 \frac{|x-x'|^\delta}{(|x-y|+|x'-y|)^{n+\delta}},~\\&\qquad
  \text{ whenever }~ 2|x-x'|\le \max(|x-y|,|x'-y|), \\
  &|K(x,y)-K(x,y')|\leq c_2 \frac{|y-y'|^\delta}{(|x-y|+|x-y'|)^{n+\delta}},~\\&\qquad
	\text{ whenever }~ 2|y-y'|\le \max(|x-y|,|x-y'|).
\end{align*}
The maximal truncation of $T$ is given by 
\[
   (T^{(\ast)} f)(y) = \sup_{\varepsilon > 0} \left|\int_{x\in\real^n, |y-x|> \varepsilon} K(y,x) f(x)\di x \right|.
\]
Then the following statements hold true.
\begin{enumerate} 
\item \label{as1}
There are linear and bounded extensions of $T$ acting in $\Lmwe$. 
\item \label{as2}
There are unique continuous and bounded extensions of $T$ and $T^{(\ast)}$ acting in $\Lciw$. 
\item \label{as3}
There is an unique linear and bounded extension of $T$ acting in \\$\Hrdw$.
\end{enumerate}
\end{Corol}
\begin{proof}
  The boundedness of $T$ in $L_{p,w}(\rn)$ for all $w\in A_p$ and all $1<p<\infty$ is covered e.g. by \cite[Thm. 9.4.6, Cor. 9.4.7]{G09} which yields Part \ref{as1} and Part \ref{as2}. The adjoint kernel of $K(x,y)$ given by $\ol{K(y,x)}$ also satisfies the required assumptions on the kernel. Hence, its corresponding operator is also bounded in $L_{p,w}(\rn)$ for all $w\in A_p$ and all $1<p<\infty$ (cf. \cite[Def. 8.1.2]{G09}) but its dual coincides with the operator $T$ (with the kernel $K(x,y)$) on $D(\rn)$ which implies Part \ref{as3}.
\end{proof}
\begin{Rem}
The well-definedness of $T^{(\ast)}$ holds by the fact that
\[
  \left|\int_{x\in\real^n, |y-x|> \varepsilon} K(y,x) f(x)\di x \right|, \qquad f\in D(\rn)
\]
is bounded for each $\ve>0$ and $y\in \rn$ as a consequence of H\"older's inequality. Hence, $(T^{(\ast)} f)(y)$ is well-defined for all $y\in\real$, but might be infinite.
The boundedness of Calder\'{o}n-Zygmund operators are crucial for an investigation of Navier-Stokes equations. We refer to the recent books of Triebel \cite{Tr12, T-HS} where Navier-Stokes equations has been investigated in the context of unweighted (Morrey-type) Besov-Triebel-Lizorkin spaces. As a further key ingredient he used wavelet characterizations for the corresponding spaces (cf. \cite{R}).
\end{Rem}
\subsubsection{H\"ormander-Mikhlin type multipliers}
We denote the Fourier transform of $f$ on $S(\rn)$ or $S'(\rn)$ 
by $\hat{f}$ and its inverse by $\check{f}$ where the normalisation of $\hat{f}$ does not matter for our purposes.
\begin{Corol} \label{HM:thm}
Let $1<p<\infty$, $-\frac{n}{p}< r<0$, $-n < \vr < -\frac{n}{p'}$ and let $w\in A_p$. Let $m\in C^n(\rn\setminus\{0\})$ which satisfies 
\[
  \sup_{R>0} \left( R^{s|\alpha|-n} \sup_{R<|x|<2R} |D^\alpha m(x)| \right)^\frac{1}{s}<\infty \quad \text{  for all  } |\alpha|\le n
\]
where $1<s\le 2$ and $|\alpha|=\alpha_1+\ldots+\alpha_n$ is a multi-index of non-negative integers $\alpha_j$.
Let $T_m$ be the operator defined by 
\begin{\eq} \label{FT_L2}
(T_m f)\,\hat{ }\,(\xi)=m(\xi) \hat{f}(\xi), ~~f\in S(\real),~ \xi\in\real. 
\end{\eq}
Then the following statements hold true.
\begin{enumerate} 
\item 
There are linear and bounded extensions of $T_m$ acting in $\Lmwe$. 
\item 
There is an unique linear and bounded extension of $T_m$ acting in \\$\Lciw$. 
\item
There is an unique linear and bounded extension of $T_m$ acting in \\$\Hrdw$.
\end{enumerate}
\end{Corol}
\begin{proof}
  The boundedness of $T_m$ in $L_{p,w}(\real)$ for all $w\in A_p$ and all $1<p<\infty$ is covered e.g. by \cite[Thm. 1]{KW79}. 
\end{proof}
\subsubsection{Marcinkiewicz multipliers}
\begin{Corol} \label{CvB1:thm}
Let $1<p<\infty$, $-\frac{1}{p}< r<0$, $-1 < \vr < -\frac{1}{p'}$ and let $w\in A_p$. Let $m$ be a bounded function which has uniformly bounded variation on each of the dyadic sets $(-2^{j+1},-2^j)\bigcup (2^j,2^{j+1})$, $j\in\ganz$, in $\real$. Let $T_m$ be the operator defined analogously to \eqref{FT_L2}.
Then the following statements hold true.
\begin{enumerate} 
\item 
There are linear and bounded extensions of $T_m$ acting in $\eLmwe$. 
\item 
There is an unique linear and bounded extension of $T_m$ acting in \\$\eLciw$. 
\item
There is an unique linear and bounded extension of $T_m$ acting in \\$\eHrdw$.
\end{enumerate}
\end{Corol}
\begin{proof}
  The boundedness of $T_m$ in $L_{p,w}(\real)$ for all $w\in A_p$ and all $1<p<\infty$ is covered e.g. by \cite[Thm. 8.35]{D01} and the reference given there (see also \cite[Thm. 5.2.2]{G08}). Moreover, we observe $T_m'=T_{m(-\cdot)}$ on $D(\rn)$ and whenever $m$ has uniformly bounded variation on each of the dyadic sets $(-2^{j+1},-2^j)\bigcup (2^j,2^{j+1})$, $j\in\ganz$, then $m(-\cdot)$ has also uniformly bounded variation on these sets.
\end{proof}
\begin{Rem}
We want to mention that $T_m$ is well-defined by \eqref{FT_L2} since $m \hat{f} \in L_2(\rn)$ for $f\in S(\real)$.
The assumption on $m$ is in particular satisfied if $m$ is a bounded function which is continuously differentiable on $(-2^{j+1},-2^j)\bigcup (2^j,2^{j+1})$, $j\in\ganz$, satisfying 
\[
  \sup_{j\in\ganz} \left[\int_{(-2^{j+1},-2^j)} |m'(\xi)|\di \xi + \int_{(2^j,2^{j+1})} |m'(\xi)|\di \xi \right] <\infty.
\]
\end{Rem}
\subsubsection{The maximal Carleson operator}
\begin{Corol} 
Let $1<p<\infty$, $-\frac{1}{p}< r<0$ and let $w\in A_p$.  
Let the maximal Carleson operator be defined as
\[ 
	C_* (f)(x)\equiv \sup_{\varepsilon>0} \sup_{\xi\in\real} \left|\int_{|x-y|>\varepsilon} \frac{f(y)e^{2\pi i\xi y }}{x-y}dy\right|
	, ~~f\in S(\real),~ \xi\in\real. 
\] 
Then there is an unique continuous and bounded extension of $C_*$ acting in $\eLciw$.
\end{Corol}
\begin{proof}
  The boundedness of $C_*$ in $L_{p,w}(\real)$ for all $w\in A_p$ and all $1<p<\infty$ is covered by \cite[Thm. 11.3.3]{G09}.
\end{proof}
\begin{Rem}
The well-definedness of $C_*$ holds because of the fact that
\[
  \left|\int_{|x-y|>\varepsilon} \frac{f(y)e^{2\pi i\xi y }}{x-y}dy\right|
\]
is bounded for each $\ve>0$ and $x\in \real$ as a consequence of H\"older's inequality. Hence, $C_* (f)(x)$ is well-defined for all $x\in\real$, but might be infinite.
\end{Rem}
\subsubsection{Commutators}
\begin{Rem}
	The commutator of a singular integral operator $T$ with a function
	$b$ is defined as
	\[
		[b,T](f) = bT(f) - T(bf) \quad\text{ for } f\in D(\rn).
	\]
	If $1<p<\infty$, $w\in A_p$ and $b\in BMO$, then $[b,T]$ is bounded in $\Lpw$.
	For detailed definitions and references we refer to \cite[page 62]{CMP11}.
  The unweighted case one finds also in \cite[Thm. 7.5.6]{G09}.
	Therefore the appropriate assertions of Corollary \ref{HM:thm} holds also for these commutators.
\end{Rem}
\subsubsection{Vector-valued boundednesses}
\begin{Rem}
All the mentioned results hold also in the vector-valued situation since Muckenhoupt weighted boundednesses imply by extrapolation its (weighted and unweighted) vector-valued boundednesses (cf. \cite[page 22]{CMP11}). Let $T_m$ be defined as in Corollary \ref{HM:thm}. Then by classical extrapolation we obtain
\[ 
  \norm{\left(\sum_{j\in\nat} |T_m f_j|^q\right)^\frac{1}{q}}{\Lpw} 
	\le c \norm{\left(\sum_{j\in\nat} |f_j|^q\right)^\frac{1}{q}}{\Lpw}
\]
for any $1<p,q<\infty$ and any $w\in A_p$. Hence, the assumption \eqref{1exa} of Theorem \ref{T3.3} is fulfilled for the following pairs of functions
\[
  \left(\left(\sum_{j\in\nat} |T_m f_j|^q\right)^\frac{1}{q}, \left(\sum_{j\in\nat} |f_j|^q\right)^\frac{1}{q}\right)
\]
for all sequences $\{f_j\}$ where $f_j \in D(\rn)$ and only a finite number of functions $f_j$ are not identically zero. 
By \eqref{2exa} we deduce then 
\[ 
  \norm{\left(\sum_{j\in\nat} |T_m f_j|^q\right)^\frac{1}{q}}{\Lmwe} 
	\le c \norm{\left(\sum_{j\in\nat} |f_j|^q\right)^\frac{1}{q}}{\Lmwe}
\]
for any $1<p,q<\infty$, any $-n/p<r<0$, any $w\in A_p$ and all sequences $\{f_j\}$ where $f_j \in D(\rn)$ and only a finite number of functions $f_j$ are not identically zero. The operator defined by 
$T (\{f_j\})\equiv \left(\sum_{j\in\nat} |T_m f_j|^q\right)^\frac{1}{q}$ can now be extended by the appropriate vector-valued duality assertions which are proved for the unweighted situation in \cite{RS14} and which lead to the appropriate results of Corollary \ref{HM:thm} in the unweighted vector-valued situation. In the same manner one would get these results also for the other operators where we dealt with $T_m$ as a model case. The vector-valued assertions of these H\"ormander-Mikhlin multipliers lead to the appropriate results in Morrey smoothness spaces of Besov-Triebel-Lizorkin type (c.f. \cite{R} for definitions). Taking those one can generalize different results on PDEs (Navier-Stokes equations, \ldots) in the recent book of Triebel \cite{T-HS} where he just dealt with model cases relying on the vector-valued boundedness of the Riesz transform. 
\end{Rem}
\section{Associated spaces}
\begin{Def} Let $1<p<\infty$, $-\frac{n}{p}< r<0$, $r+\vr=-n$ and let $w\in A_p$.
\begin{enumerate} 
\item
The \textit{associated space of }$\Lmwe$ is given by the norm 
\begin{\eq*} 
  \norm{f}{\Lmwe^{\$}}\equiv\sup_{g}\int_{\rn} |f(x){g}(x) |\di x 
\end{\eq*}
where the supremum is taken over all $\left\|g|\Lmwe\right\|\leq 1$ and $f\in L_1^{\mathrm{loc}}(\rn)$. 
\item
The \textit{associated space of }$\Hrdw$ is given by the norm 
\begin{\eq*} 
  \norm{f}{\Hrdw^{\$}}\equiv\sup_{g}\int_{\rn} |f(x){g}(x) |\di x 
\end{\eq*}
where the supremum is taken over all $\left\|g|\Hrdw\right\|\leq 1$ and $f\in L_1^{\mathrm{loc}}(\rn)$. 
\end{enumerate} 
\end{Def}
We recover results of \cite{MST16}.
\begin{Corol}
Let $1<p<\infty$, $-\frac{n}{p}< r<0$, $r+\vr=-n$ and let $w\in A_p$.
Then 
\[
  \norm{f}{\Lmwe^{\$}} \cong \norm{f}{\Hrdw}
	\text{ and }
	\norm{f}{\Hrdw^{\$}} \cong \norm{f}{\Lmwe}
\]
where $f\in L_1^{\mathrm{loc}}(\rn)$. 
\end{Corol}
\begin{proof}
By \eqref{58:GGG} follows
\[
  \norm{f}{\Lmwe^{\$}} \le c \norm{f}{\Hrdw},
	\norm{f}{\Hrdw^{\$}} \le c \norm{f}{\Lmwe}
	.
\]
Let us assume now that $f\in \Lmwe^{\$}$. 
We observe that 
\[
\norm{f}{\Lmwe^{\$}}=\sup_{g}\int_{\rn} f(x){g}(x) \di x=\sup_{g}\left|\int_{\rn} f(x){g}(x) \di x \right|
\]
where the supremum is taken over all $\left\|g|\Lmwe\right\|\leq 1$. Thus,
\[
  \sup_{g}\left|\int_{\rn} f(x){g}(x) \di x \right|\le \norm{f}{\Lmwe^{\$}}
\] 
where the supremum is taken over all $g\in D(\rn)$ with $\left\|g|\Lmwe\right\|\leq 1$. But then the left-hand side of the latter inequality coincides with the norm of a functional on $\Lciw$. By \eqref{represDual} and \eqref{represDualNorm} we achieve
\[
  \norm{f}{\Hrdw} \le c \sup_{g}\left|\int_{\rn} f(x){g}(x) \di x \right|\le c\norm{f}{\Lmwe^{\$}}.
\] 
Using \eqref{87:GG} and \eqref{MWGH} we deduce analogously 
\[
  \norm{f}{\Lmwe}\le c\norm{f}{\Hrdw^{\$}}.
\] 
\end{proof}
\section*{Acknowledgements}
We thank Javier Duoandikoetxea (UPV/EHU) for useful hints with respect to Section 5 dealing with extrapolation in Morrey spaces. 
%


\begin{thebibliography}{xxxxxxxxx}

\bibitem[AX04]{AX04}	\textsc{Adams, D. R. and Xiao , J.}: Nonlinear potential analysis on Morrey spaces and their
capacities. Indiana Univ. Math. J. \textbf{53}: No.6, 1629-1663 (2004).

\bibitem[AX12]{AX12}	\textsc{Adams, D. R. and Xiao , J.}: Morrey spaces in harmonic analysis. Ark. Mat. \textbf{50}, 201-230 (2012).





\bibitem[Ad15]{Ad15}	\textsc{Adams, D. R.}:
Morrey spaces. 
Lecture Notes in Applied and Numerical Harmonic Analysis. Cham: Birkh\"auser/Springer 2015.

\bibitem[Alv96]{Alv96} \textsc{Alvarez, J.}: Continuity of Calder\'{o}n-Zygmund type operators on the predual of a Morrey space. In: Clifford algebras in analysis and
related topics, CRC Press, Boca Raton, 309-319 (1996).


\bibitem[BRV99]{BRV99} \textsc{Blasco, O.; Ruiz, A.; Vega, L.}:
Non interpolation in Morrey-Campanato and block spaces. 
Ann. Sc. Norm. Super. Pisa, Cl. Sci., IV. Ser. \textbf{28}, No.1, 31-40 (1999).

\bibitem[CMP04]{CMP04} \textsc{Cruz-Uribe, D. V.; Martell, J.M.; P\'erez, C.}:
Extrapolation from $A_\infty$ weights and applications. 
J. Funct. Anal. 213, No. \textbf{2}, 412-439 (2004).

\bibitem[CMP11]{CMP11} \textsc{Cruz-Uribe, D. V.; Martell, J.M.; P\'erez, C.}:
Weights, extrapolation and the theory of Rubio de Francia.
Operator Theory: Advances and Applications 215. Basel: Birkh\"auser (2011).

\bibitem[CGCMP06]{CGCMP06} \textsc{Curbera, G. P.; García-Cuerva, J.; Martell, J.M.; P\'erez, C.}:
Extrapolation with weights, rearrangement-invariant function spaces, modular inequalities and applications to singular integrals. 
Adv. Math. \textbf{203}, No. 1, 256-318 (2006).

\bibitem[DYZ98]{Y98} \textsc{Ding, Y.; Yang, D.; Zhou, Z.}:
Boundedness of sublinear operators and commutators on Morrey spaces.
Yokohama Math. J. \textbf{46}, No.1, 15-27 (1998).


\bibitem[Duo01]{D01} \textsc{Duoandikoetxea, J.}:
Fourier analysis. Transl. from the Spanish and revised by David Cruz-Uribe. 
Graduate Studies in Mathematics. \textbf{29}. Providence, RI: American Mathematical Society (2001).

\bibitem[Duo11]{D11} \textsc{Duoandikoetxea, J.}: Extrapolation of weights revisited: new proofs and sharp bounds. J. Funct. Anal. \textbf{260}, No. 6, 1886-1901 (2011).

\bibitem[Duo13]{D13} \textsc{Duoandikoetxea, J.}: Lecture Notes Spring School \textit{Function Spaces and Inequalities}, Paseky 2013, ISBN 978-80-7378-233-7, pp. 150.

\bibitem[ET96]{ET96} \textsc{Edmunds, D. E. and Triebel, H.}:
Function spaces, entropy numbers, differential operators. Paperback reprint
of the hardback edition 1996. 
Cambridge Tracts in Mathematics \textbf{120}. Cambridge: Cambridge University Press.

\bibitem[GM13]{GM13} \textsc{Gogatishvili, A. and Mustafayev, R.Ch.}: 
New pre-dual space of Morrey space. 
J. Math. Anal. Appl. {\bfseries 397}, No. 2, 678-692 (2013).


\bibitem[Gra08]{G08}	\textsc{Grafakos, L.}:
Classical Fourier analysis. 2nd ed. Grad. Texts in Math. \textbf{249}. Springer, New York 2008.

\bibitem[Gra09]{G09}	\textsc{Grafakos, L.}:
Modern Fourier analysis. 2nd ed. Grad. Texts in Math. \textbf{250}. Springer, New York 2009.

\bibitem[GR01]{GR01}	\textsc{Gr\"oger, K. and Recke, L.}:
Preduals of Campanato spaces and Sobolev-Campanato spaces: A general construction. 
Math. Nachr. \textbf{230}, 45-72 (2001).

\bibitem[GAKS11]{G11}	\textsc{Guliyev, V. S.; Aliyev, S. S.; Karaman, T.; Shukurov, P. S.}:
Boundedness of Sublinear Operators and Commutators on Generalized Morrey Spaces.
Integr. Equ. Oper. Theory \textbf{71}, 327-355 (2011).

\bibitem[Gul12]{Gul12}	\textsc{Guliyev, V. S.}:
Generalized weighted Morrey spaces and higher order commutators of sublinear operators. 
Eurasian Math. J. \textbf{3}, No. 3, 33-61 (2012).

\bibitem[ISY14]{ISY14} \textsc{Izumi, T.; Sato, E.; Yabuta, K.}: Remarks on a subspace of Morrey spaces,
Tokyo J. Math. \textbf{37}, No. 1, 185-197 (2014).

\bibitem[IST15]{IST15} \textsc{Izumi, T.; Sawano, Y.; Tanaka, H.}:
Littlewood-Paley theory for Morrey spaces and their preduals. 
Rev. Mat. Complut. \textbf{28}, No. 2, 411-447 (2015).


\bibitem[Kal98]{Kal98} \textsc{Kalita, E.A.}: Dual Morrey spaces. Dokl. Akad. Nauk {\bfseries 361} (1998), 447-449 (Russian); Engl. transl.: Dokl. Math. 
{\bfseries 58}, 85-87 (1998).


\bibitem[KGS14]{KGS14} \textsc{Karaman, T.; Guliyev, V. S.; Serbetci, A.}
Boundedness of sublinear operators generated by Calder\'on-Zygmund operators on generalized weighted Morrey spaces. 
An. Stiint. Univ. Al. I. Cuza Iasi, Ser. Noua, Mat. \textbf{60}, No. 1, 227-244 (2014).

\bibitem[KS09]{KS09} \textsc{Komori, Y. and Shirai, S.}: Weighted Morrey spaces and a singular integral operator. Math. Nachr. \textbf{282}, No. 2, 219-231 (2009).

\bibitem[KW79]{KW79} \textsc{Kurtz, D. S.and Wheeden, R. L.}:
Results on weighted norm inequalities for multipliers. 
Trans. Am. Math. Soc. \textbf{255}, 343-362 (1979).

\bibitem[Lon84]{Lon84} \textsc{Long, R.}: 
The spaces generated by blocks. 
Sci. Sin., Ser. A \textbf{27}, 16-26 (1984).


\bibitem[MST16]{MST16}	\textsc{Masty\l o, M.; Sawano, Y.; Tanaka, H.}:
Morrey-type Space and Its K\"othe Dual Space. 
Bulletin of the Malaysian Mathematical Sciences Society, DOI 10.1007/s40840-016-0382-7, p. 1-18 (2016).

\bibitem[Mus12]{Mu12}	\textsc{Mustafayev, Rza Ch.}: 
On boundedness of sublinear operators in weighted Morrey spaces.
Azerb. J. of Math. \textbf{2}, 66-79 (2012).

\bibitem[Nak94]{N94}	\textsc{Nakai, E.}: Hardy-Littlewood maximal operator, singular integral operators and the
Riesz potentials on generalized Morrey spaces. 
Math. Nachr. \textbf{166}, 95-103 (1994).

\bibitem[Pee66]{Pee66} \textsc{Peetre, J.}: 
On convolution operators leaving $L^{p, \lambda}$ spaces invariant.
Ann. Mat. Pura Appl. {\bfseries 72}, 295-304 (1966).

\bibitem[PT15]{PT15} \textsc{Poelhuis, J. and Torchinsky, A.}:
Weighted local estimates for singular integral operators. 
Trans. Am. Math. Soc. \textbf{367}, No. 11, 7957-7998 (2015).


\bibitem[RSS13]{RSS} \textsc{Rafeiro, H.; Samko, N.; Samko, S.}: Morrey-Campanato spaces: an overview. 
Operator Theory: Advances and Applications \textbf{228}, 293-323 (2013).

\bibitem[Ros12]{R} \textsc{Rosenthal, M.}: Local means, wavelet bases and wavelet isomorphisms in Besov-Morrey and Triebel-Lizorkin-Morrey spaces. Math. Nachr. \textbf{286}, No. 1, 59-87 (2013).

\bibitem[Ros13]{Ros13} \textsc{Rosenthal, M.}: Mapping properties of operators in Morrey spaces and wavelet isomorphisms in related Morrey smoothness spaces. PhD-Thesis, Jena, 2013.

\bibitem[RS14]{RS14} \textsc{Rosenthal, M. and Schmei\ss{}er, H.-J.}: On the boundedness of singular integrals in Morrey spaces and its preduals, J. Fourier Anal. Appl., 22(2), 462-490, 2016. (DOI 10.1007/s00041-015-9427-9) \& Erratum (DOI 10.1007/s00041-015-9438-6), 22(2), p. 491, 2016. 
\\Preprint on http://arxiv.org/pdf/1409.0679v1.pdf, p. 32 (2014).

\bibitem[RT13]{RT13} \textsc{Rosenthal, M. and Triebel, H.}: Calder\'{o}n-Zygmund operators in Morrey spaces. Rev. Mat. Complut. \textbf{27}, 1-11 (2014).

\bibitem[RT14]{RoT14_2} \textsc{Rosenthal, M. and Triebel, H.}: Morrey spaces, their duals and preduals. Rev. Mat. Complut. \textbf{28}, 1-30 (2015).


\bibitem[Rub82]{Ru1} \textsc{Rubio de Francia, J. L.}:
Factorization and extrapolation of weights. 
Bull. Am. Math. Soc., New Ser. \textbf{7}, 393-395 (1982).

\bibitem[Rub84]{Ru2} \textsc{Rubio de Francia, J. L.}:
Factorization theory and $A_p$ weights. 
Am. J. Math. \textbf{106}, 533-547 (1984).

\bibitem[Sam13]{Sam13} \textsc{Samko, N.}:
On a Muckenhoupt-type condition for Morrey spaces. Mediterr. J. Math. \textbf{10}, No. 2, 941-951 (2013).

\bibitem[Sam14]{Sam14} \textsc{Samko, N.}:
On two-weight estimates for the maximal operator in local Morrey spaces. 
Int. J. Math. \textbf{25}, No. 11, Article ID 1450099, 8 p. (2014).

\bibitem[ST05]{Saw05} \textsc{Sawano, Y. and Tanaka, H.}:
Morrey spaces for non-doubling measures. 
Acta Math. Sin., Engl. Ser. \textbf{21}, No. 6, 1535-1544 (2005).

\bibitem[ST09]{ST09} \textsc{Sawano, Y. and Tanaka, H.}:
Predual spaces of Morrey spaces with non-doubling measures. 
Tokyo J. Math. \textbf{32}, 471-486 (2009).

\bibitem[SFZ13]{SFZ13} \textsc{Shi, S.; Fu, Z.; Zhao, F.}:
Estimates for operators on weighted Morrey spaces and their applications to nondivergence elliptic equations. 
J. Inequal. Appl. 2013, Article ID 390, 16 p., (2013).

\bibitem[Ste93]{Ste93} \textsc{Stein, E.M.}: Harmonic Analysis, real-variable methods, orthogonality, and oscillatory integrals. 
Princeton, NJ: Princeton University Press, 1993.

\bibitem[Tri78]{T78}	\textsc{Triebel, H.}:            
Interpolation theory. Function spaces. Differential operators. 
Deutscher Verlag des Wissenschaften, Berlin 1978.


\bibitem[Tri13]{Tr12} \textsc{Triebel, H.}: Local function spaces, heat and Navier-Stokes equations. Z\"urich: European Mathematical Society, 2013.    

\bibitem[Tri14]{T-HS}
\textsc{Triebel, H.}:
Hybrid function spaces, heat and Navier-Stokes equations. Z\"urich: European Mathematical Society, 2014. 



\bibitem[Wan16]{Wan16} \textsc{Wang, H.}: Boundedness of $\theta$-type Calder\'on-Zygmund operators and commutators in the generalized weighted Morrey spaces. http://arxiv.org/abs/1603.03912v1.

\bibitem[Zor86]{Zor86} \textsc{Zorko, C.T.}: Morrey spaces. Proc. Amer. Math. Soc. {\bfseries 98}, 586-592 (1986).


\end{thebibliography}
\end{document}